\documentclass{article}
\usepackage[a4paper, margin=1in]{geometry}
\usepackage{setspace}
\usepackage{xcolor}
\usepackage{soul}
\usepackage{graphicx}
\usepackage{amsmath, amsthm, amssymb}
\usepackage{mathtools}
\usepackage{array}
\usepackage{multirow}
\usepackage{hyperref} 
\usepackage{enumerate}
\usepackage[style=numeric, sorting=none]{biblatex}
\newtheorem{theorem}{Theorem}[section]
\newtheorem{proposition}[theorem]{Proposition}
\newtheorem{lemma}[theorem]{Lemma}
\newtheorem{corollary}[theorem]{Corollary}
\theoremstyle{definition}
\newtheorem{definition}[theorem]{Definition}
\newtheorem{example}[theorem]{Example}
\newtheorem{remark}[theorem]{Remark}
\usepackage{float}
\usepackage{booktabs}
\title{Ev-Zagreb index of Graphs: Inequalities and Graph Operations}
\usepackage{authblk}
\author{Amitariddhi Sinha\thanks{riddhi@tezu.ernet.in}}
\author{Somnath Paul\thanks{som@tezu.ernet.in}}
\affil{Department of Applied Sciences, Tezpur University, Assam, India}
\date{}

\begin{document}

\maketitle

\begin{abstract}
    The ev-degree of an edge \(uv\) of a graph \(G\) is the number of vertices in \(N_G[u]\cup N_G[v]\), and the ev-Zagreb index \(M^{ev}(G)\) is the sum of the squares of the ev-degrees of the edges of \(G\). In this chapter, we prove that \(M^{ev}(G)=HM(G)-2T_1(G)+T_2(G)\), where \(HM\) is the hyper-Zagreb index and \(T_1,T_2\) are triangle-weighted edge sums, and that \(M^{ev}\) is monotone under taking subgraphs. We give sharp lower and upper bounds for \(M^{ev}\) in terms of the order, the size, the Zagreb indices and the number of triangles, and characterise all extremal graphs. We also obtain a Tur\'an-type bound for \(K_{r+1}\)-free graphs and a Nordhaus--Gaddum type inequality. We derive exact formulas for the ev-Zagreb index of the complement, which involve the number of \(4\)-cycles, and of the line graph. Finally, we compute \(M^{ev}\) for the subdivision, semitotal and total graphs, the join, Cartesian product, composition, corona and tensor product, and several subdivision-based joins and coronas.
\end{abstract}

\section{Introduction}

Let \(G\) be a finite simple graph with vertex set \(V(G)\), edge set \(E(G)\), \(n=|V(G)|\) vertices and \(m=|E(G)|\) edges. We write \(d_G(u)\) for the degree of a vertex \(u\), \(N_G(u)\) for its neighbourhood and \(N_G[u]=N_G(u)\cup \{u\}\) for its closed neighbourhood.

\subsection{Topological indices: origins and purpose}

A \emph{topological index} is a real number attached to a graph that is invariant under graph isomorphism. The idea comes from chemistry. A molecule is modelled by its \emph{molecular graph}: atoms other than hydrogen become vertices and chemical bonds become edges. A single number computed from this graph is then compared with a measured property of the substance, such as its boiling point, heat of formation or biological activity. When the agreement is good, the index can predict the property for compounds that have not been synthesised or measured. This methodology is known as QSPR/QSAR (quantitative structure--property/activity relationships). The descriptors now in use number in the thousands~\cite{todeschini2000handbook}.

The first index of this kind was introduced by Wiener in 1947~\cite{wiener1947structural}. His \emph{path number}, the sum of the distances between all pairs of carbon atoms, is today called the \emph{Wiener index}. It was used to estimate the boiling points of paraffins. The term “topological index” was coined by Hosoya in 1971~\cite{hosoya1971topological}. Two families followed: distance-based indices, descended from the Wiener index, and degree-based indices, which depend only on the degrees of the vertices. The present chapter concerns the second family.

\subsection{Degree-based indices and the Zagreb indices}

The degree-based family begins with the \emph{Zagreb indices}. They appeared in 1972, when Gutman and Trinajsti\'c~\cite{gutman1972graph} studied how the total \(\pi\)-electron energy of conjugated molecules depends on molecular structure. The sum of squared vertex degrees occurred naturally in their approximate formula. The second quantity was introduced in 1975~\cite{gutman1975graph}:
\[
M_1(G)=\sum_{u\in V(G)} d_G(u)^2=\sum_{uv \in E(G)}\bigl(d_G(u)+d_G(v)\bigr),
\qquad
M_2(G)=\sum_{uv\in E(G)}d_G(u)\,d_G(v).\]
In the same year Randi\'c~\cite{randic1975characterization} proposed his \emph{branching index}
\(R(G)=\displaystyle\sum_{uv\in E(G)}\bigl(d_G(u)d_G(v)\bigr)^{-1/2}\), which became the most widely used molecular descriptor of its time.

Most degree-based indices share one form. They are edge sums 
\begin{equation}\label{eq:vdb}
    TI_\varphi(G)=\sum_{uv\in E(G)} \varphi\bigl(d_G(u),d_G(v)\bigr)
\end{equation}
for a symmetric function \(\varphi\)~\cite{gutman2013degree}.
Table~\ref{tab:vdb} lists the indices that appear in this chapter, together with some frequently studied relatives.

\begin{table}[ht]
    \centering
    \renewcommand{\arraystretch}{1.35}
    \begin{tabular}{@{}llll@{}}
    \toprule
    Index & \(\varphi(x,y)\) & Year & Reference \\
    \midrule
        First Zagreb \(M_1\) & \(x+y\) & 1972 & \cite{gutman1972graph} \\
        Second Zagreb \(M_2\) & \(xy\) & 1975 & \cite{gutman1975graph} \\
        Randi\'c \(R\) & \((xy)^{-1/2}\) & 1975 & \cite{randic1975characterization}\\
        Harmonic \(H\) & \(2/(x+y)\) & 1987 & \cite{fajtlowicz1987conjectures} \\
        General Randi\'c  \(R_\alpha\) & \((xy)^{\alpha}\) & 1998 & \cite{bollobas1998graphs}\\
        Atom--bond connectivity \(ABC\) & \(\sqrt{(x+y-2)/(xy)}\) & 1998 & \cite{estrada1998atom}\\
        Geometric--arithmetic \(GA\) & \(2\sqrt{xy}/(x+y)\) & 2009 & \cite{vukicevic2009topological}\\
        Sum-connectivity \(\chi\) & \((x+y)^{-1/2}\) & 2009 & \cite{zhou2009novel}\\
        Hyper-Zagreb \(HM\) & \((x+y)^2\) & 2013 & \cite{shirdel2013hyper}\\
        Forgotten \(F\) & \(x^2+y^2\) & 2015 & \cite{furtula2015forgotten}\\
        Sombor \(SO\) & \(\sqrt{x^2+y^2}\) & 2021 & \cite{gutman2021geometric}\\
        \bottomrule
    \end{tabular}
    \caption{Some degree-based topological indices of the form~\eqref{eq:vdb}}
    \label{tab:vdb}
\end{table}

The \emph{forgotten index} has an unusual history. The quantity \(F(G)=\sum_{u\in V(G)}d_G(u)^3\) already appeared in the 1972 paper~\cite{gutman1972graph}, next to \(M_1\). It was then ignored for more than four decades, until Furtula and Gutman~\cite{furtula2015forgotten} returned to it and named it in 2015. Since \((x+y)^2=(x^2+y^2)+2xy\), the three indices satisfy
\[
HM(G)=F(G)+2M_2(G)
.\]
This identity is used repeatedly in what follows.

\paragraph{Known results.}
The Zagreb indices have been studied far more than any other degree-based index. The surveys~\cite{nikolic2003zagreb, gutman2004first, borovicanin2017bounds} cover this literature. We recall a few results of the kind proved later in the chapter for the ev-Zagreb index.

\begin{itemize}
    \item \emph{Bounds in terms of \(n\) and \(m\).} By the Cauchy--Schwarz inequality, \((2m)^2=\bigl(\sum_{u\in V(G)}d_G(u)\bigr)^2\le n\,M_1(G)\). Hence \(M_1(G)\ge 4m^2/n\), with equality if and only if \(G\) is regular. In the other direction, de Caen~\cite{decaen1998upper} proved \(M_1(G)\le m\bigl(\tfrac{2m}{n-1}+n-2\bigr)\). Many refinements of both bounds are collected in~\cite{borovicanin2017bounds}.
    \item \emph{Extremal trees.} Among trees of order \(n\), both \(M_1\) and \(M_2\) are minimised by the path \(P_n\) and maximised by the star \(K_{1,n-1}\)~\cite{gutman2004first}. 
    \item \emph{Comparing \(M_1\) and \(M_2\).} Hansen and Vuki\v{c}evi\'c~\cite{hansen2007comparing} asked whether \(M_1(G)/n\le M_2(G)/m\) for every graph. The inequality holds for chemical graphs and for trees~\cite{vukicevic2007comparing}, but fails in general.
\end{itemize}

\paragraph{Structural meaning.}
Topological indices also carry combinatorial information. For example, \[
\sum_{u\in V(G)}\binom{d_G(u)}{2}
\] 
counts the paths of length two in \(G\). Hence \(M_1(G)=2m+2p_2(G)\), where \(p_2(G)\) is the number of such paths. A second example comes from the line graph \(L(G)\). The degree of an edge \(uv\) in \(G\), that is, its degree as a vertex of \(L(G)\), is \(d_G(u)+d_G(v)-2\). Replacing vertex degrees by these edge degrees in \(M_1\) and \(M_2\) gives the \emph{reformulated Zagreb indices} of Mili\v{c}evi\'c, Nikoli\'c and Trinajsti\'c~\cite{milicevic2004reformulated}. The first of them is exactly \(M_1(L(G))\). This idea, building an index on a notion of degree other than the vertex degree, is the one that leads to the ev-Zagreb index.

\subsection{Edge-vertex domination and the ev-degree}
Vertex-edge and edge-vertex domination were introduced by Peters in his 1986 thesis~\cite{peters1986theoretical}. They were developed further by Lewis~\cite{lewis2007vertex} and by Boutrig, Chellali, Haynes and Hedetniemi~\cite{boutrig2016vertex}.

\begin{definition}
    An edge \(e=uv\) \emph{ev-dominates} a vertex \(w\) if \(w\in N_G[u]\cup N_G[v]\), that is, if \(w\) is an end-vertex of \(e\) or is adjacent to one. A vertex \(w\) \emph{ve-dominates} an edge \(e\) if \(e\) has an end-vertex in \(N_G[w]\). The \emph{ev-degree} of \(e\) and the \emph{ve-degree} of \(w\) are 
    \[
    \deg^{ev}_G(e)=\bigl|N_G[u]\cup N_G[v]\bigr|,
    \qquad
    \deg^{ve}_G(w)=\bigl|\{e\in E(G): e \text{ has an end-vertex in } N_G[w]\}\bigr|.
    \]
\end{definition}

These degrees were introduced as graph parameters by Chellali, Haynes, Hedetniemi and Lewis~\cite{chellali2017ve}. Their paper studied ev-regular and ve-regular graphs and posed several open problems on the regularity and irregularity of these degrees. One of the problems was solved by Horoldagva, Das and Selenge~\cite{horoldagva2019ve}.

The ev-degree has a simple form. Let \(t_G(uv)=|N_G(u)\cap N_G(v)|\) be the number of triangles containing the edge \(uv\). By inclusion--exclusion, 
\begin{equation}\label{eq:evdeg}
     \deg^{ev}_G(uv)=d_G(u)+d_G(v)-t_G(uv).
\end{equation}
The ev-degree is therefore the degree sum of the end-vertices, corrected for triangles. It coincides with \(d_G(u)+d_G(v)\) precisely when \(uv\) lies in no triangle. The two notions of degree are linked by a double count. A vertex \(w\) is ev-dominated by \(e\) if and only if \(w\) ve-dominates \(e\), so 
\begin{equation}\label{eq:sum}
    \sum_{w\in V(G)}\deg^{ve}_G(w)=\sum_{e\in E(G)}\deg^{ev}_G(e)=M_1(G)-3t(G),
\end{equation}
where \(t(G)\) is the number of triangles in \(G\). The second equality was proved by Chellali \textit{et al.}~\cite{chellali2017ve}.

\subsection{The ev-Zagreb index}

Identity~\eqref{eq:sum} shows that the first-order sum of ev-degrees is determined by \(M_1\) and the triangle count. This led Ediz in 2017~\cite{ediz2017predicting} to consider the second-order sum.

\begin{definition}[\cite{ediz2017predicting}]
    The \emph{ev-degree Zagreb index} (\emph{ev-Zagreb index}) of \(G\) is 
    \[
    M^{ev}(G)=\sum_{e\in E(G)}\deg^{ev}_G(e)^2 .
    \]
\end{definition}

Ediz~\cite{ediz2017predicting} denoted this index by \(S(G)\) (not to be confused with the subdivision graph \(S(G)\) of Section~\ref{sec:prelim}). In the same paper he introduced several ve-degree indices: the first ve-degree Zagreb \(\alpha\)-index \(\displaystyle\sum_{w\in V(G)}\deg^{ve}_G(w)^2\), the \(\beta\)-index \(\displaystyle\sum_{uv\in E(G)}\bigl(\deg^{ve}_G(u)+\deg^{ve}_G(v)\bigr)\), the second ve-degree Zagreb index \(\displaystyle\sum_{uv\in E(G)}\bigl(\deg^{ve}_G(u)\deg^{ve}_G(v)\bigr)\) and the ve-degree Randi\'c index. He proved the following results for the ev-Zagreb index.
\begin{itemize}
    \item If \(G\) is connected and triangle-free, then \(M^{ev}(G)=F(G)+2M_2(G)=HM(G)\). This follows from~\eqref{eq:evdeg}.
    \item For every tree \(T\) of order \(n\), \(16n-30\le M^{ev}(T)\le n^2(n-1)\). The lower bound is attained by the path and the upper bound by the star.
    \item For connected graphs of order \(n\), the lower bound \(16n-30\) persists, and \(M^{ev}(G)\le \frac{1}{2}n^3(n-1)\), with equality for \(K_n\).
\end{itemize}

\paragraph{Variations.}
Replacing the square in \(M^{ev}\) by other functions gives a family of ev-degree indices, parallel to the family in Table~\ref{tab:vdb}. Examples are the ev-degree forgotten index \(\displaystyle\sum_{e}\deg^{ev}_G(e)^3\), the modified ev-Zagreb index \(\displaystyle\sum_{e}\deg^{ev}_G(e)^{-2}\) and the ev-degree Randi\'c index \(\displaystyle\sum_{e}\deg^{ev}_G(e)^{-1/2}\). These appear, together with their ve-degree counterparts, in the computations of Delen \textit{et al.}~\cite{delen2022ve}. \c{S}ahin and Ediz~\cite{sahin2018ev} introduced the multiplicative \emph{ev- and ve-degree Narumi--Katayama indices}, of which the first is \(\displaystyle\prod_e\deg^{ev}_G(e)\). They determined extremal graphs for these indices and gave basic properties of the ev- and ve-degree Zagreb indices. 

\paragraph{Usability.}
Ediz~\cite{ediz2017predicting} tested the new indices against four physicochemical properties of the octane isomers. These are a standard benchmark set in chemical graph theory. For the entropy of the octanes, the ev-Zagreb index gave the strongest correlation among all the indices tested, with \(|r|\approx 0.96\). For the enthalpies of vaporisation, however, the classical Randi\'c index remained the better predictor. Delen \textit{et al.}~\cite{delen2022ve} computed ev- and ve-degree indices of the drug molecule fenofibrate. As is usual for a new descriptor, these studies are small. The chemical value of the index for larger data sets, and its structure-sensitivity in the sense of Furtula, Gutman and Dehmer~\cite{furtula2013structure}, are still to be assessed.

From the mathematical side, the interest of \(M^{ev}\) lies in what it adds to the classical indices. On triangle-free graphs it reduces to \(HM\). On all other graphs, by~\eqref{eq:evdeg}, it also detects how the edges are arranged into triangles. Indices built from the degree sequence alone cannot detect this.

\begin{example}
    The cycle \(C_6\) and the graph \(2K_3\) are both \(2\)-regular of order \(6\). Hence every vertex-degree-based index of the form~\eqref{eq:vdb} has the same value on both graphs. In particular, 
    \(M_1=M_2=24\), \(F=48\) and \(HM=96\) for each of them. The ev-Zagreb index separates them:
    \[
    M^{ev}(C_6)=6\cdot 4^2=96,
    \qquad
    M^{ev}(2K_3)=6\cdot 3^2=54.
    \]
\end{example}

The ev-Zagreb index is also an edge-domination quantity. The value \(\deg^{ev}_G(e)=n\) means that the two end-vertices of \(e\) form a dominating set of \(G\). This connects \(M^{ev}\) with domination theory, a link that the classical indices do not have.

\subsection{Graph operations}
Many graphs of interest are built from smaller graphs by standard constructions. These include the complement, the line graph and the subdivision graph, and also binary operations such as the join and the Cartesian product. Molecular graphs of polymers, nanotubes and nanotori are of this kind. A formula that gives an index of the resulting graph in terms of invariants of the original graphs therefore has two uses. It gives the value of the index for a whole family at once, and it shows how the index behaves under the construction.

This line of work is well developed for the classical degree-based indices. Khalifeh, Yousefi-Azari and Ashrafi~\cite{khalifeh2009first} computed the first and second Zagreb indices of the join, Cartesian product, composition, disjunction and symmetric difference of graphs. Shirdel, Rezapour and Sayadi~\cite{shirdel2013hyper} did the same for the hyper-Zagreb index, and De, Nayeem and Pal~\cite{de2016findex} for the forgotten index. Eliasi and Taeri~\cite{eliasi2009four} introduced four unary operations: the subdivision graph \(S(G)\), the edge-semitotal graph \(Q(G)\), the vertex-semitotal graph \(R(G)\) and the total graph \(T(G)\). They also introduced the corresponding \(F\)-sums of two graphs. The Zagreb indices of these graphs were computed by Deng, Sarala, Ayyaswamy and Balachandran~\cite{deng2016zagreb}, and the forgotten index by Akhter and Imran~\cite{akhter2017computing}. For the complement and the line graph, Liu, Ali, Malik, Siddiqui and Imran~\cite{liu2019reformulated} computed the reformulated Zagreb indices.

For the ev-Zagreb index, the problem has an additional feature. A vertex-degree-based index of a derived graph depends only on the degrees of the derived graph. These are usually simple functions of the degrees in \(G\). By~\eqref{eq:evdeg}, the ev-degree also depends on the number of triangles through each edge. A construction that creates or destroys triangles therefore brings new invariants of \(G\) into the formula. Two examples from this chapter show the effect.
\begin{itemize}
    \item \emph{Complement.} A non-edge \(uv\) of \(G\) becomes an edge of \(\overline{G}\) with ev-degree \(n-|N_G(u)\cap N_G(v)|\). Summing the squares over all non-edges requires the number of pairs of common neighbours, that is, the number \(q(G)\) of \(4\)-cycles of \(G\). This dependence is genuine. The graphs \(C_8\) and \(2C_4\) have the same degrees and no triangles, and so the same value of every vertex-degree-based index and of \(M^{ev}\). Their complements nevertheless have different ev-Zagreb indices, \(1160\) and \(1168\), because \(q(C_8)=0\) and \(q(2C_4)=2\).
    \item \emph{Line graph.} Every vertex of degree \(d\) in \(G\) produces a clique \(K_d\) in \(L(G)\). Hence \(L(G)\) has many triangles even when \(G\) has none, and \(M^{ev}(L(G))\) involves degree sums over neighbourhoods. These are the neighbourhood first Zagreb index \(NM_1(G)\) and the quantity \(\Pi(G)=\displaystyle\sum_{uv \in E(G)}d_G(u)d_G(v)\bigl(d_G(u)+d_G(v)\bigr)\), together with \(\displaystyle\sum_{u}d_G(u)^4\).
\end{itemize}

\subsection{Contents of this chapter}

Section~2 collects the notation and the auxiliary results used later. It also proves the identity
\[
M^{ev}(G)=HM(G)-2T_1(G)+T_2(G),
\]
where
\[
T_1(G)=\sum_{uv\in E(G)}t_G(uv)\bigl(d_G(u)+d_G(v)\bigr),
\qquad
T_2(G)=\sum_{uv\in E(G)}t_G(uv)^2,
\]
which extends Ediz's triangle-free formula to all graphs.

Section~3 develops the theory. We prove that \(M^{ev}\) is monotone under taking subgraphs. We then give lower and upper bounds in terms of the classical indices, \(n\), \(m\) and the triangle count, and characterise all equality cases. These lead to the extremal values recalled above, a Tur\'an-type upper bound for \(K_{r+1}\)-free graphs and a Nordhaus--Gaddum type inequality.

We then turn to graph operations. We give exact formulas for the ev-Zagreb index of the complement \(\overline{G}\) and of the line graph \(L(G)\). The first is in terms of \(n\), \(m\), \(M_1\), the triangle count, \(T_2\) and the \(4\)-cycle count of \(G\). The second is in terms of \(M_1\), \(M_2\), \(F\), \(Z_4\), \(\Pi\), \(NM_1\), \(T_1\) and the triangle count of \(G\). We evaluate both formulas for standard families: paths, cycles, complete and complete bipartite graphs, wheels, the cube and the Petersen graph. Next we compute \(M^{ev}\) for the subdivision graph \(S(G)\), the edge-semitotal graph \(Q(G)\), the vertex-semitotal graph \(R(G)\) and the total graph \(T(G)\), and for the join, Cartesian product, composition, corona and tensor product of two graphs. Finally, we treat the subdivision-based joins and coronas of Barik, Kalita, Pati and Sahoo~\cite{barik2018spectra}.

\section{Preliminaries}\label{sec:prelim}

This section fixes the notation used throughout the chapter. It also collects the counting identities and known results on which the proofs in Section~3 rely, and defines the graph operations studied there.

\subsection{Graphs and basic notation}

All graphs are finite and simple. For a graph \(G\) we write \(n(G)=|V(G)|\) for its order and \(m(G)=|E(G)|\) for its size. When only one graph is under consideration we simply write \(n\) and \(m\). An edge with end-vertices \(u\) and \(v\) is written \(uv\). For \(u\in V(G)\), \(N_G(u)\) is the (open) neighbourhood of \(u\), \(N_G[u]=N_G(u)\cup \{u\}\) is its closed neighbourhood and \(d_G(u)=|N_G(u)|\) is its degree. A vertex of degree \(0\) is \emph{isolated}, and \(G\) is \emph{\(r\)-regular} if \(d_G(u)=r\) for every \(u\in V(G)\).

A subgraph \(H\) of \(G\) is \emph{spanning} if \(V(H)=V(G)\). For a graph \(F\), \(G\) is \emph{\(F\)-free} if it has no subgraph isomorphic to \(F\); a \(K_3\)-free graph is called \emph{triangle-free}. The \emph{girth} of \(G\) is the length of a shortest cycle in \(G\). A set \(D\subseteq V(G)\) is a \emph{dominating set} if every vertex of \(G\) lies in \(D\) or is adjacent to a vertex of \(D\).

The \emph{disjoint union} of \(G\) and \(H\) is denoted by \(G\cup H\), and \(kG\) denotes the disjoint union of \(k\) copies of \(G\). The \emph{complement} \(\overline{G}\) of \(G\) has vertex set \(V(G)\), two distinct vertices being adjacent in \(\overline{G}\) if and only if they are non-adjacent in \(G\). We write \(\overline{m}=m(\overline{G})=\binom{n}{2}-m\).

We use the following standard graphs.
\begin{itemize}
    \item \(K_n\), \(P_n\) and \(C_n\) are the complete graph, the path and the cycle on \(n\) vertices; \(\overline{K_n}
    \) is the edgeless graph of order \(n\). 
    \item \(K_{a,b}\) is the complete bipartite graph with parts of sizes \(a\) and \(b\); \(K_{1,n-1}\) is the star of order \(n\).
    \item For \(r\geq 2\), the \emph{Tur\'an graph} \(T_r(n)\) is the complete \(r\)-partite graph of order \(n\) whose part sizes differ by at most one. We write \(t_r(n)=m(T_r(n))\); in particular \(t_2(n)=\lfloor n^2/4\rfloor\).
    \item The \emph{wheel} \(W_n=C_n+K_1\) consists of a cycle \(C_n\) (the \emph{rim}) and one further vertex adjacent to all rim vertices. Thus \(W_n\) has \(n+1\) vertices and \(2n\) edges.
    \item The \emph{complete split graph} is \(CS(a,b)=K_a+\overline{K_b}\) where \(+\) denotes the join (Definition~\ref{def:binary}).
    \item \(Q_3\) is the \(3\)-dimensional cube and \(P\) denotes the Petersen graph; both are \(3\)-regular and triangle-free, and \(P\) has girth \(5\).
\end{itemize}

\begin{definition}\label{def:falsetwins}
    Let \(H\) be a subgraph of \(G\), and put \(N_H(x)=\emptyset\) for \(x\in V(G)\setminus V(H)\). Two distinct vertices \(u,v \in V(G)\) are \emph{false twins of \(H\)} if \(N_H(u)=N_H(v)\).
\end{definition}
For a statement \(\mathcal{P}\), the \emph{Iverson bracket} \([\mathcal{P}]\) equals \(1\) if \(\mathcal{P}\) is true and \(0\) otherwise. For example, for distinct vertices \(a,b\) of \(G\), \([ab\in E(G)]\) is \(1\) when \(a\) and \(b\) are adjacent and \(0\) otherwise.

\subsection{Common neighbours, triangles and \(4\)-cycles}

The ev-degree of an edge depends on how many neighbours its end-vertices share. We therefore introduce the following local counts.

\begin{definition}\label{def:counts}
    Let \(G\) be a graph.
    \begin{enumerate}[(i)]
        \item For distinct \(u,v\in V(G)\), \(c_G(u,v)=|N_G(u)\cap N_G(v)|\) is the number of common neighbours of \(u\) and \(v\).
        \item For an edge \(uv\), \(t_G(uv)=c_G(u,v)\); this is the number of triangles of \(G\) that contain the edge \(uv\).
        \item For a vertex \(u\), \(t_G(u)\) is the number of triangles of \(G\) that contain \(u\).
        \item \(t(G)\) is the number of triangles of \(G\), and \(q(G)\) is the number of \(4\)-cycles of \(G\) (as subgraphs).
    \end{enumerate}
\end{definition}

\begin{lemma}\label{lem:counting}
    For every graph \(G\) of order \(n\) and size \(m\):
    \begin{enumerate}[(a)]
        \item \(\displaystyle\sum_{e\in E(G)}t_G(e)=3t(G)\),\quad
        \(\displaystyle\sum_{u\in V(G)}t_G(u)=3t(G)\),\quad and\quad
  \(\displaystyle\sum_{e\ni u}t_G(e)=2t_G(u)\) for every \(u\in V(G)\), where the last sum runs over the edges incident with \(u\);
  \item \(\displaystyle\sum_{\{u,v\}}c_G(u,v)=\sum_{w\in V(G)}\binom{d_G(w)}{2}
  =\frac{M_1(G)}{2}-m\);
  \item \(\displaystyle\sum_{\{u,v\}}\binom{c_G(u,v)}{2}=2q(G)\),
    \end{enumerate}
    where the sums in (b) and (c) run over all \(\binom{n}{2}\) unordered pairs of distinct vertices.
\end{lemma}

\begin{proof}
    \begin{enumerate}[(a)]
        \item Every triangle has three edges and three vertices, and contains exactly two of the edges incident with each of its vertices.
        \item Both sides count the paths \(uwv\) of length two in \(G\): on the left by their end-vertices, on the right by their middle vertex \(w\). The last equality follows from \(\displaystyle\sum_{w\in V(G)}d_G(w)^2=M_1(G)\) and \(\displaystyle\sum_{w\in V(G)}d_G(w)=2m\).
        \item A pair \(\{u,v\}\) together with two of its common neighbours \(a,b\) determines the \(4\)-cycle \(uavb\). Conversely, a \(4\)-cycle \(uavb\) arises in this way exactly twice: from the pair \(\{u,v\}\) with \(\{a,b\}\), and from the pair \(\{a,b\}\) with \(\{u,v\}\).      
    \end{enumerate}
\end{proof}
\subsection{Degree-based invariants}
Table~\ref{tab:invariants} lists the invariants of \(G\) that occur in our formulas. For \(u\in V(G)\) we put 
\[
S_u=\sum_{w\in N_G(u)}d_G(w)
\qquad \text{ and } \qquad Q_u=\sum_{w\in N_G(u)}d_G(w)^2.
\]

\begin{table}[ht]
    \centering
    \renewcommand{\arraystretch}{1.25}
    \begin{tabular}{@{}lll@{}}
    \toprule
    Symbol & Definition & Name \\
    \midrule
       \(M_1(G)\)  & \(\displaystyle\sum_{u\in V(G)}d_G(u)^2=\displaystyle\sum_{uv\in E(G)}\bigl(d_G(u)+d_G(v)\bigr)\) 
         & first Zagreb index\\
         \(M_2(G)\) & \(\displaystyle\sum_{uv\in E(G)}d_G(u)d_G(v)\) & second Zagreb index\\
         \(F(G)\) & \(\displaystyle\sum_{u\in V(G)}d_G(u)^3=\displaystyle\sum_{uv\in E(G)}\bigl(d_G(u)^2+d_G(v)^2\bigr)\) & forgotten index\\
          \(HM(G)\) & \(\displaystyle\sum_{uv\in E(G)}\bigl(d_G(u)+d_G(v)\bigr)^2=F(G)+2M_2(G)\) & hyper-Zagreb index\\
          \(Z_4(G)\) & \(\displaystyle\sum_{u\in V(G)}d_G(u)^4\) \\
          \(\Pi(G)\) & \(\displaystyle\sum_{uv\in E(G)}d_G(u)d_G(v)\bigl(d_G(u)+d_G(v)\bigr)\)\\
          \(NM_1(G)\) & \(\displaystyle\sum_{u\in V(G)}S_u^2\), \(S_u=\displaystyle\sum_{w\in N_G(u)}d_G(w)\) & first neighbourhood Zagreb index\\
          \(t(G)\) & number of triangles in \(G\); \(\displaystyle\sum_{e\in E(G)}t_G(e)=3t(G)\)                  \\
          \(T_1(G)\) & \(\displaystyle\sum_{uv\in E(G)}t_G(uv)\bigl(d_G(u)+d_G(v)\bigr)\)                     \\
          \(T_2(G)\)  & \(\displaystyle\sum_{e\in E(G)}t_G(e)^2\)                     \\
          \(q(G)\)    & number of \(4\)-cycles in \(G\)                   \\   
         \bottomrule
    \end{tabular}
    \caption{Invariants used in this chapter.}
    \label{tab:invariants}
\end{table}

\begin{lemma}\label{lem:degreeids}
   For every graph \(G\):
   \begin{enumerate}[(a)]
       \item \(HM(G)=F(G)+ 2M_2(G)\);
       \item \(\displaystyle\sum_{u\in V(G)}S_u=M_1(G)\) and \(\displaystyle\sum_{u\in V(G)}d_G(u)\,S_u=2M_2(G)\);
       \item \(\displaystyle\sum_{u\in V(G)}Q_u=F(G)\) and 
       \(\displaystyle\sum_{u\in V(G)}d_G(u)\,Q_u=\Pi(G)\);
       \item \(T_1(G)=2\displaystyle\sum_{u\in V(G)}t_G(u)\,d_G(u)\).
   \end{enumerate}
\end{lemma}

\begin{proof}
   \begin{enumerate}
       \item [(a)]        Expand \(\bigl(d_G(u)+d_G(v)\bigr)^2\) and use the second expression for \(F(G)\) in Table~\ref{tab:invariants}.
     \item [(b)--(c)]  A double sum \(\sum_{u}\sum_{w\in N_G(u)}g(u,w)\) runs over every edge \(uw\) twice, once in each direction. Hence \(\sum_u S_u=\sum_{uw\in E(G)}
\bigl(d_G(u)+d_G(w)\bigr)=M_1(G)\) and 
\(\sum_u d_G(u)S_u=\sum_{uw\in E(G)}2d_G(u)d_G(w)=2M_2(G)\). In the same way, \(\sum_u Q_u=\sum_{uw\in E(G)}\bigl(d_G(u)^2+d_G(w)^2\bigr)=F(G)\) and \(\sum_u d_G(u)Q_u=\sum_{uw\in E(G)}\bigl(d_G(u)d_G(w)^2+d_G(w)d_G(u)^2\bigr)
=\Pi(G)\).
\item [(d)] Write \(T_1(G)=\sum_u d_G(u)\sum_{e\ni u}t_G(e)\) and apply Lemma~\ref{lem:counting}(a).
   \end{enumerate} 
\end{proof}

For an \(r\)-regular graph of order \(n\), every invariant in Table~\ref{tab:invariants} reduces to a function of \(n\), \(r\) and \(t(G)\): \begin{equation}\label{eq:regular}
    m=\tfrac{nr}{2}, \quad M_1=nr^2, \quad M_2=\tfrac{nr^3}{2}, \quad F=nr^3, \quad HM=2nr^3, \quad Z_4=\Pi=NM_1=nr^4, \quad T_1=6r\,t(G).
\end{equation}

\subsection{Known results}

We shall use the Cauchy--Schwarz inequality in the form 
\begin{equation}\label{eq:CS}
    \Bigl(\sum_{i=1}^{k}x_i\Bigr)^2\leq k\sum_{i=1}^{k}x_i^2
\qquad(x_1,\dots,x_k\in\mathbb{R})
\end{equation}
with equality if and only if \(x_1=\dots=x_k\), and Jensen's inequality in the form \(f(x)+f(y)\geq 2f\bigl(\tfrac{x+y}{2}\bigr)\) for a strictly convex function \(f\), with equality if and only if \(x=y\).

\begin{lemma}\label{lemma1}
    Let \(G\) be a graph with \(n\geq 1\) vertices and \(m\) edges. Then 
    \[
    M_1(G)\geq \frac{4m^2}{n},
    \]
    with equality if and only if \(G\) is regular.
\end{lemma}
\begin{proof}
    Apply~\eqref{eq:CS} to the \(n\) vertex degrees: \((2m)^2=\bigl(\displaystyle\sum_{u\in V(G)}d_G(u)\bigr)^2\leq n\displaystyle\sum_{u\in V(G)}d_G(u)^2=nM_1(G)\). Equality holds if and only if all degrees are equal.
\end{proof}

Lemma~\ref{lemma1} holds for all graphs, connected or not. This matters because we apply it to disconnected graphs and to complements.

\begin{theorem}[Tur\'an~\cite{turan1941external}]\label{thm:turan}
Let \(r\geq 2\). Every \(K_{r+1}\)-free graph \(G\) of order \(n\) satisfies \(m(G)\leq t_r(n)\), with equality if and only if \(G\cong T_r(n)\).
\end{theorem}

\begin{theorem}[\cite{gao2017sharp}]\label{thm:HMtrees}
    For every tree \(T\) of order \(n\geq 3\),
    \[
    16n-30\leq HM(T)\leq n^2(n-1).
    \]
    Equality holds on the left if and only if \(T\cong P_n\), and on the right if and only if \(T\cong K_{1,n-1}\).
\end{theorem}

\subsection{The ev-degree and ev-Zagreb index}

Recall from the Section 1, that an edge \(e=uv\) \emph{ev-dominates} a vertex \(w\) if \(w\in N_G[u]\cup N_G[v]\), and that the \emph{ev-degree} of \(e\) is 
\[
\deg^{ev}_G(e)=\bigl|N_G[u]\cup N_G[v]\bigr|=\bigl|N_G(u)\cup N_G(v)\bigr|.
\]
The two unions coincide because \(v\in N_G(u)\) and \(u\in N_G(v)\). In particular \(\deg^{ev}_G(e)\geq 2\) for every edge \(e\). A graph is \emph{ev-regular} if all its edges have the same ev-degree. The \emph{ev-Zagreb index} of \(G\) is 
\[
M^{ev}(G)=\sum_{e\in E(G)}\deg^{ev}_G(e)^2,
\]
with \(M^{ev}(G)=0\) when \(G\) has no edges. The next lemma collects the elementary properties of the ev-degree.

\begin{lemma}\label{l2}
    For every edge \(e=uv\) of a graph \(G\) of order \(n\),
    \[
    \deg^{ev}_G(e)=d_G(u)+d_G(v)-t_G(uv)
=n-\bigl|V(G)\setminus\bigl(N_G[u]\cup N_G[v]\bigr)\bigr|.
    \]
    Consequently,
    \[
    \max\{d_G(u),d_G(v)\}+1\leq\deg^{ev}_G(e)\leq\min\{n,\,d_G(u)+d_G(v)\}.
    \]
    Moreover:
    \begin{enumerate}[(a)]
        \item the lower bound is attained if and only if \(N_G[u]\subseteq N_G[v]\) or \(N_G[v]\subseteq N_G[u]\);
        \item \(\deg^{ev}_G(e)=d_G(u)+d_G(v)\) if and only if \(e\) lies in no triangle;
        \item \(\deg^{ev}_G(e)=n\) if and only if \(\{u,v\}\) is a dominating set
  of \(G\).
    \end{enumerate}
\end{lemma}

\begin{proof}
By inclusion--exclusion,
$|N_G(u)\cup N_G(v)|=d_G(u)+d_G(v)-|N_G(u)\cap N_G(v)|
=d_G(u)+d_G(v)-t_G(uv)$. The second expression is the definition read in
$V(G)$. Since $\deg^{ev}_G(e)=|N_G[u]\cup N_G[v]|$ is at least
$\max\{|N_G[u]|,|N_G[v]|\}=\max\{d_G(u),d_G(v)\}+1$, the lower bound
follows, with equality exactly when one closed neighbourhood contains the
other. The upper bounds and (b), (c) are immediate from the two
expressions.
\end{proof}
 
\begin{proposition}[\cite{chellali2017ve}]\label{p1}
For every graph $G$,
\[
\sum_{e\in E(G)}\deg^{ev}_G(e)=M_1(G)-3t(G).
\]
\end{proposition}
 
\begin{proof}
Sum the formula of Lemma~\ref{l2} over all edges and use
Lemma~\ref{lem:counting}(a).
\end{proof}
 
The following identity expresses $M^{ev}$ through the hyper-Zagreb index
and the triangle data of $G$. It is the basis of most results in
Section~3.
 
\begin{theorem}[Fundamental identity]\label{thm2}
For every graph $G$,
\[
M^{ev}(G)=HM(G)-2T_1(G)+T_2(G)=F(G)+2M_2(G)-2T_1(G)+T_2(G).
\]
In particular, $M^{ev}(G)\leq HM(G)$, with equality if and only if $G$ is
triangle-free.
\end{theorem}
 
\begin{proof}
By Lemma~\ref{l2},
\begin{align*}
M^{ev}(G)&=\sum_{uv\in E(G)}\bigl(d_G(u)+d_G(v)-t_G(uv)\bigr)^2\\
&=\sum_{uv\in E(G)}\bigl(d_G(u)+d_G(v)\bigr)^2
-2\sum_{uv\in E(G)}t_G(uv)\bigl(d_G(u)+d_G(v)\bigr)
+\sum_{uv\in E(G)}t_G(uv)^2,
\end{align*}
which is $HM(G)-2T_1(G)+T_2(G)$; the second form follows from
Lemma~\ref{lem:degreeids}(a). Since $0\leq\deg^{ev}_G(uv)\leq
d_G(u)+d_G(v)$ for every edge, $M^{ev}(G)\leq HM(G)$, and by
Lemma~\ref{l2}(b) equality holds if and only if no edge lies in a triangle.
\end{proof}
 
The theorem extends the identity $M^{ev}(G)=F(G)+2M_2(G)$ of
Ediz~\cite{ediz2017predicting} for connected triangle-free graphs to all
graphs. The values for the standard graphs follow directly from the
definition.
 
\begin{example}\label{ex:standard}
\begin{enumerate}[(i)]
  \item Every edge of $K_n$ has ev-degree $n$, so
  $M^{ev}(K_n)=\frac12 n^3(n-1)$. More generally, every edge of $kK_r$ has
  ev-degree $r$, so $M^{ev}(kK_r)=k\binom{r}{2}r^2$.
  \item In a complete multipartite graph, the end-vertices of every edge
  form a dominating set. Hence $M^{ev}(K_{a,b})=ab(a+b)^2$ and
  $M^{ev}(T_r(n))=n^2t_r(n)$ by Lemma~\ref{l2}(c). In particular
  $M^{ev}(K_{1,n-1})=n^2(n-1)$.
  \item For $n\geq 3$, the two end edges of $P_n$ have ev-degree $3$ and the
  remaining $n-3$ edges have ev-degree $4$. Hence $M^{ev}(P_n)=16n-30$.
  \item For $n\geq 4$, every edge of $C_n$ has ev-degree $4$, so
  $M^{ev}(C_n)=16n$, while $M^{ev}(C_3)=27$.
  \item The cube $Q_3$ is ev-regular with ev-degree $6$, so
  $M^{ev}(Q_3)=12\cdot 36=432$.
\end{enumerate}
\end{example}

\subsection{Graph operations}
 
We now define the graph operations considered in Section~3.
Table~\ref{tab:orders} lists the order and size of each resulting graph.
 
\paragraph{Unary operations.}
The \emph{line graph} $L(G)$ has vertex set $E(G)$, two edges of $G$ being
adjacent in $L(G)$ if and only if they share an end-vertex. Thus the degree
of $uv\in E(G)$ as a vertex of $L(G)$ is $d_G(u)+d_G(v)-2$, and the edges of
$L(G)$ correspond to the paths of length two in $G$.
 
The next four graphs were introduced by Eliasi and Taeri~\cite{eliasi2009four}
(see also~\cite{akhter2017computing}). Each has vertex set $V(G)\cup E(G)$.
The elements of $V(G)$ are called \emph{old vertices} and those of $E(G)$
\emph{new vertices}.
 
\begin{definition}\label{def:unary}
Let $G$ be a graph.
\begin{enumerate}[(i)]
  \item The \emph{subdivision graph} $S(G)$ has an edge $ue$ for every old
  vertex $u$ and every new vertex $e$ incident with $u$ in $G$. Thus
  $S(G)$ is obtained by inserting a new vertex into every edge of $G$.
  \item The \emph{edge-semitotal graph} $Q(G)$ is obtained from $S(G)$ by
  joining two new vertices whenever the corresponding edges of $G$ are
  adjacent, i.e.\ by adding the edges of $L(G)$.
  \item The \emph{vertex-semitotal graph} $R(G)$ is obtained from $S(G)$ by
  joining two old vertices whenever they are adjacent in $G$, i.e.\ by
  adding the edges of $G$.
  \item The \emph{total graph} $T(G)$ is obtained from $S(G)$ by adding both
  the edges of $G$ and the edges of $L(G)$.
\end{enumerate}
\end{definition}
 
\paragraph{Binary operations.}
Let $G$ and $H$ be graphs with disjoint vertex sets.
 
\begin{definition}\label{def:binary}
\begin{enumerate}[(i)]
  \item The \emph{join} $G+H$ is obtained from $G\cup H$ by joining every
  vertex of $G$ to every vertex of $H$. More generally, the join
  $G_1+G_2+\dots+G_k$ of $k$ pairwise vertex-disjoint graphs is obtained
  from $G_1\cup\dots\cup G_k$ by joining every vertex of $G_i$ to every
  vertex of $G_j$ for all $i\neq j$.
  \item The \emph{Cartesian product} $G\otimes H$ has vertex set
  $V(G)\times V(H)$. Two vertices $(x,y)$ and $(x',y')$ are adjacent if and
  only if either $x=x'$ and $yy'\in E(H)$, or $y=y'$ and $xx'\in E(G)$.
  \item The \emph{composition} (lexicographic product) $G\circ H$ has vertex
  set $V(G)\times V(H)$. Two vertices $(x,y)$ and $(x',y')$ are adjacent if
  and only if either $xx'\in E(G)$, or $x=x'$ and $yy'\in E(H)$.
  \item The \emph{corona} $G\odot H$ is obtained from one copy of $G$ and
  $n(G)$ copies of $H$ by joining the $i$-th vertex of $G$ to every vertex
  of the $i$-th copy of $H$, for $1\leq i\leq n(G)$.
  \item The \emph{tensor} (Kronecker) \emph{product} $G\times H$ has vertex
  set $V(G)\times V(H)$. Two vertices $(x,y)$ and $(x',y')$ are adjacent if
  and only if $xx'\in E(G)$ and $yy'\in E(H)$.
\end{enumerate}
\end{definition}
 
\paragraph{Operations based on the subdivision graph.}
The following variants of the join and the corona are taken from the survey
of Barik, Kalita, Pati and Sahoo~\cite{barik2018spectra}. Let $G$, $H$ and
$K$ be pairwise vertex-disjoint graphs. We start from one copy of $S(G)$ and
use its old vertices $V(G)$ and new vertices $E(G)$.
 
\begin{definition}\label{def:subdivision-ops}
\begin{enumerate}[(i)]
  \item The \emph{subdivision vertex join} $G\dot{\vee}H$ is obtained from
  $S(G)\cup H$ by joining every old vertex of $S(G)$ to every vertex of $H$.
  \item The \emph{subdivision edge join} $G\veebar H$ is obtained from
  $S(G)\cup H$ by joining every new vertex of $S(G)$ to every vertex of $H$.
  \item The \emph{subdivision vertex-edge join}
  $G^{S}\triangleright(H^{V}\cup K^{E})$ is obtained from $S(G)\cup H\cup K$
  by joining every old vertex of $S(G)$ to every vertex of $H$ and every new
  vertex of $S(G)$ to every vertex of $K$.
  \item The \emph{subdivision vertex corona} $G^{(S)}\odot H$ is obtained
  from $S(G)$ and $n(G)$ copies of $H$ by joining the $i$-th old vertex of
  $S(G)$ to every vertex of the $i$-th copy of $H$.
  \item The \emph{subdivision edge corona} $G^{(S)}\ominus H$ is obtained
  from $S(G)$ and $m(G)$ copies of $H$ by joining the $j$-th new vertex of
  $S(G)$ to every vertex of the $j$-th copy of $H$.
  \item The \emph{subdivision vertex neighbourhood corona}
  $G^{(S)}\boxdot H$ is obtained from $S(G)$ and $n(G)$ copies of $H$ by
  joining every neighbour (in $S(G)$) of the $i$-th old vertex to every
  vertex of the $i$-th copy of $H$.
  \item The \emph{subdivision edge neighbourhood corona} $G^{(S)}\boxminus H$
  is obtained from $S(G)$ and $m(G)$ copies of $H$ by joining both
  neighbours (in $S(G)$) of the $j$-th new vertex to every vertex of the
  $j$-th copy of $H$.
  \item The \emph{subdivision double corona} $G^{(S)}\circ(H,K)$ is obtained
  from $S(G)$, $n(G)$ copies of $H$ and $m(G)$ copies of $K$ by joining the
  $i$-th old vertex to every vertex of the $i$-th copy of $H$ and the
  $j$-th new vertex to every vertex of the $j$-th copy of $K$.
  \item The \emph{subdivision double neighbourhood corona}
  $G^{S}\bullet(H,K)$ is obtained from $S(G)$, $n(G)$ copies of $H$ and
  $m(G)$ copies of $K$ by joining the neighbours of the $i$-th old vertex to
  every vertex of the $i$-th copy of $H$ and the neighbours of the $j$-th
  new vertex to every vertex of the $j$-th copy of $K$.
\end{enumerate}
\end{definition}
 
\begin{table}[ht]
\centering
\renewcommand{\arraystretch}{1.3}
\small
\begin{tabular}{@{}lll@{}}
\toprule
Graph & Order & Size \\
\midrule
$\overline{G}$ & $n_1$ & $\binom{n_1}{2}-m_1$ \\
$L(G)$ & $m_1$ & $\tfrac12 M_1(G)-m_1$ \\
$S(G)$ & $n_1+m_1$ & $2m_1$ \\
$Q(G)$ & $n_1+m_1$ & $m_1+\tfrac12 M_1(G)$ \\
$R(G)$ & $n_1+m_1$ & $3m_1$ \\
$T(G)$ & $n_1+m_1$ & $2m_1+\tfrac12 M_1(G)$ \\
\midrule
$G+H$ & $n_1+n_2$ & $m_1+m_2+n_1n_2$ \\
$G\otimes H$ & $n_1n_2$ & $n_1m_2+n_2m_1$ \\
$G\circ H$ & $n_1n_2$ & $n_2^2m_1+n_1m_2$ \\
$G\odot H$ & $n_1(1+n_2)$ & $m_1+n_1m_2+n_1n_2$ \\
$G\times H$ & $n_1n_2$ & $2m_1m_2$ \\
\midrule
$G\dot{\vee}H$ & $n_1+m_1+n_2$ & $2m_1+m_2+n_1n_2$ \\
$G\veebar H$ & $n_1+m_1+n_2$ & $2m_1+m_2+m_1n_2$ \\
$G^{S}\triangleright(H^{V}\cup K^{E})$ & $n_1+m_1+n_2+n_3$
  & $2m_1+m_2+m_3+n_1n_2+m_1n_3$ \\
$G^{(S)}\odot H$ & $n_1+m_1+n_1n_2$ & $2m_1+n_1m_2+n_1n_2$ \\
$G^{(S)}\ominus H$ & $n_1+m_1+m_1n_2$ & $2m_1+m_1m_2+m_1n_2$ \\
$G^{(S)}\boxdot H$ & $n_1+m_1+n_1n_2$ & $2m_1+n_1m_2+2m_1n_2$ \\
$G^{(S)}\boxminus H$ & $n_1+m_1+m_1n_2$ & $2m_1+m_1m_2+2m_1n_2$ \\
$G^{(S)}\circ(H,K)$ & $n_1+m_1+n_1n_2+m_1n_3$
  & $2m_1+n_1m_2+m_1m_3+n_1n_2+m_1n_3$ \\
$G^{S}\bullet(H,K)$ & $n_1+m_1+n_1n_2+m_1n_3$
  & $2m_1+n_1m_2+m_1m_3+2m_1n_2+2m_1n_3$ \\
\bottomrule
\end{tabular}
\caption{Order and size of the graphs obtained by the operations of this
section, where $n_1=n(G)$, $m_1=m(G)$, $n_2=n(H)$, $m_2=m(H)$, $n_3=n(K)$
and $m_3=m(K)$.}
\label{tab:orders}
\end{table}
 
\begin{remark}\label{rem:proofs}
In Section~3 the edge set of a derived graph is split into classes
$E_1,E_2,\dots$ according to the type of its end-vertices (old--old,
old--new, inside a copy, and so on). For each class we compute the degrees
of the end-vertices and the number of triangles through the edge. The
ev-degree then follows from Lemma~\ref{l2}, and $S_i$ denotes the
contribution $\sum_{e\in E_i}\deg^{ev}(e)^2$ of the class $E_i$. Sums of
the form $\sum_{e\in E(H)}\deg^{ev}_H(e)$, which occur whenever the edges of
a factor $H$ reappear with a shifted ev-degree, are evaluated by
Proposition~\ref{p1}.
\end{remark}
\section{Results}
\subsection{Inequalities and extremal graphs}
\begin{lemma}\label{l1}
    Let \(H\) be a subgraph of \(G\). Then \(\deg^{ev}_H(e)=\deg^{ev}_G(e)\) for every \(e\in E(H)\) if and only if every edge of \(E(G)\setminus E(H)\) joins two false twins of \(H\).
\end{lemma}
\begin{proof}
    Suppose \(\deg^{ev}_H(e)=\deg^{ev}_G(e)\) for every \(e\in E(H)\). Let \(uv\) be an edge of \(E(G)\setminus E(H)\) and \(w\in N_H(u)\). Then the condition \(\deg^{ev}_H(uw)=\deg^{ev}_G(uw)\) implies \(N_H(u)\cup N_H(w)=N_G(u)\cup N_G(w)\). Since \(v\in N_G(u)\subseteq N_G(u)\cup N_G(w)\), \(v\in N_H(u)\cup N_H(w)\). Since \(uv\notin E(H)\), \(v\notin N_H(u)\), so \(v\in N_H(w)\) i.e.\ \(w\in N_H(v)\). Hence \(N_H(u)\subseteq N_H(v)\). The other inclusion, \(N_H(v)\subseteq N_H(u)\) holds, similarly. Hence \(N_H(u)=N_H(v)\).

    Conversely, assume that every edge of \(E(G)\setminus E(H)\) joins two false twins of \(H\). Let \(e=uv\in E(H)\). Since \(H\) is a subgraph of \(G\), \(N_H(u)\cup N_H(v)\subseteq N_G(u)\cup N_G(v)\). Let \(w\in N_G(u)\cup N_G(v)\). If \(w\in N_H(u)\cup N_H(v)\), then there is nothing to prove. If \(w\in N_G(u)\setminus N_H(u)\) then \(uw\in E(G)\setminus E(H)\), which implies \(u\) and \(w\) are false twins of \(H\), hence \(N_H(u)=N_H(w)\). Now \(v\in N_H(u)\) implies \(v\in N_H(w)\) i.e.\ \(w\in N_H(v)\subseteq N_H(u)\cup N_H(v)\). Similarly if \(w\in N_G(v)\setminus N_H(v)\), then \(w\in N_H(u)\subseteq N_H(u)\cup N_H(v)\). Hence \( N_H(u)\cup N_H(v)=N_G(u)\cup N_G(v)\) and \(\deg^{ev}_H(e)=\deg^{ev}_G(e)\).
\end{proof}
\begin{theorem}[Monotonicity]\label{thm3} If \(H\) is a subgraph of \(G\), then \(M^{ev}(H)\leq M^{ev}(G)\), with equality if and only if \(G=H\cup sK_1\) where \(s=|V(G)\setminus V(H)|\geq 0\). 
\end{theorem}
\begin{proof}
    For an edge \(e=uv\in E(H)\), \(\deg^{ev}_H(e)\leq \deg^{ev}_{G}(e)\), since \(N_H(u)\cup N_H(v)\subseteq N_G(u)\cup N_G(v)\). Therefore, \(M^{ev}(H)=\displaystyle\sum_{e\in E(H)}\deg^{ev}_H(e)^2\leq \displaystyle\sum_{e\in E(H)}\deg^{ev}_G(e)^2 \leq \displaystyle\sum_{e\in E(G)}\deg^{ev}_G(e)^2=M^{ev}(G) \). 

    If \(G=H\cup sK_1\) where \(s=|V(G)\setminus V(H)|\geq 0\) then 
    \(E(G)\setminus E(H)=\emptyset\), which implies \(M^{ev}(H)=\displaystyle\sum_{e\in E(H)}\deg^{ev}_H(e)^2= \displaystyle\sum_{e\in E(H)}\deg^{ev}_G(e)^2 =\displaystyle\sum_{e\in E(G)}\deg^{ev}_G(e)^2=M^{ev}(G)\). The first equality holds by Lemma \ref{l1}, whose condition is vacuously satisfied and the second equality holds because both the sums run over the same edge-set. 
    
    Conversely, assume \(M^{ev}(H)= M^{ev}(G)\). Then both the inequalities above should be equalities. Since \(\deg^{ev}_H(e)^2\leq  \deg^{ev}_{G}(e)^2\) for every edge \(e\in E(H)\), the first equality gives \(\deg^{ev}_H(e)=  \deg^{ev}_{G}(e)\) for every edge \(e\in E(H)\), and so, by Lemma \ref{l1}, every edge of \(E(G)\setminus E(H)\) joins two false twins of \(H\). The second equality gives \(\displaystyle\sum_{e\in E(G)\setminus E(H)}\deg^{ev}_G(e)^2=0\), which in turn gives \( E(G)\setminus E(H)=\emptyset\), because \(\deg^{ev}_G(e)\geq 2\) for all edges. Thus the false-twin condition holds only vacuously. Now \(E(H)=E(G)\) implies every edge of \(G\) has both ends in \(V(H)\). So the remaining \(s=|V(G)\setminus V(H)|\) vertices are isolated in \(G\). Hence \(G=H\cup sK_1\).
\end{proof}
\begin{theorem}\label{thm4}
    Let \(G\) be a graph with \(n\) vertices and \(m\geq 1\) edges. Then 
    \begin{enumerate}[(a)]
        \item \(\displaystyle\sum_{uv\in E(G)}(\max\{d_G(u),d_G(v)\}+1)^2\leq M^{ev}(G)\leq \min\{n^2m, HM(G)\}\). Equality holds on the left if and only if the closed neighbourhoods of the end-points of each edge are nested; \(M^{ev}(G)=HM(G)\) if and only if \(G\) is triangle-free, and \(M^{ev}(G)=n^2m\) if and only if the end-points of each edge form a dominating set.
        \item \(M^{ev}(G)\geq \frac{1}{4}HM(G)+ M_1(G)+ m\). Equality holds if and only if every component of \(G\) is a complete graph.
        \item \(M^{ev}(G)\geq m\left(1+\frac{2m}{n}\right)^2\). Equality holds if and only if \(G\cong kK_r\) where \(k\geq 1\) and \(r\geq 2\).
        \item \(M^{ev}(G)\geq \frac{(M_1(G)-3t(G))^2}{m}\). Equality holds if and only if all the edges of \(G\) have the same ev-degree.
    \end{enumerate}
\end{theorem}
\begin{proof}
    \begin{enumerate}[(a)]
        \item Proof follows from Lemma \ref{l2}.
        \item By part (a) and the fact that \(\max\{d_G(u), d_G(v)\}\geq \frac{1}{2}(d_G(u)+d_G(v))\) for every edge \(uv\) in \(G\), \(M^{ev}(G)\geq \displaystyle\sum_{uv\in E(G)}(\max\{d_G(u),d_G(v)\}+1)^2\geq \displaystyle\sum_{uv\in E(G)}(\frac{1}{2}(d_G(u)+d_G(v))+1)^2=\frac{1}{4}HM(G)+ M_1(G)+ m\). Equality holds if and only if for every edge \(uv\in E(G)\), \(d_G(u)=d_G(v)\) and \(N_G[u]\subseteq N_G[v]\) or \(N_G[v]\subseteq N_G[u]\) i.e.\ \(N_G[u]= N_G[v]\) for every edge \(uv\in E(G)\). In a connected component, this propagates along edges, so all closed neighbourhoods are equal, which makes the component complete. Conversely, in a complete component every edge satisfies \(d_G(u)=d_G(v)\) and \(N_G[u]=N_G[v]\), so equality holds.
        \item From part (b) and the Cauchy--Schwarz inequality, \(M^{ev}(G)\geq \displaystyle\sum_{uv\in E(G)}(\frac{1}{2}(d_G(u)+d_G(v))+1)^2\geq \frac{1}{m}\left(\displaystyle\sum_{uv\in E(G)}\left(\frac{1}{2}(d_G(u)+d_G(v))+1\right)\right)^2=\frac{1}{m}\left(\frac{M_1(G)}{2}+m\right)^2\). By Lemma \ref{lemma1}, \(M_1(G)\geq \frac{4m^2}{n}\). Together, we have \(M^{ev}(G)\geq m\left(1+\frac{2m}{n}\right)^2\). 

        Equality holds if and only if all three inequalities are equalities. The first equality holds if and only if every component of \(G\) is complete, by part (b). The equality in the Cauchy--Schwarz step holds if and only if \(d_G(u)+d_G(v)\) is constant for every edge \(uv\in E(G)\). By Lemma \ref{lemma1}, \(M_1(G)=\frac{4m^2}{n}\) if and only if \(G\) is a regular graph. A regular graph has constant \(d_G(u)+d_G(v)\), so the third condition implies the second one. If \(G\) is \(d\)-regular and every component is complete, every component is \(K_{d+1}\). Since \(m\geq 1\), \(d\geq 1\). So \(G\cong kK_r\) where \(k\geq 1\) and \(r=d+1\geq 2\). Conversely, if \(G\cong kK_r\), then \(n=kr\), \(m=k\binom{r}{2}\), \(1+\frac{2m}{n}=r\) and each edge has ev-degree equal to \(r\), so \(M^{ev}(G)= m\left(1+\frac{2m}{n}\right)^2\).
        \item By the Cauchy--Schwarz inequality, \(\left(\displaystyle\sum_{uv\in E(G)}\deg^{ev}_G(uv)\right)^2\leq \left(\displaystyle\sum_{uv\in E(G)}\deg^{ev}_G(uv)^2\right)\left(\displaystyle\sum_{uv\in E(G)}1\right)\). By Proposition \ref{p1}, \(\displaystyle\sum_{uv\in E(G)}\deg^{ev}_G(uv)=M_1(G)-3t(G)\). So, \((M_1(G)-3t(G))^2\leq mM^{ev}(G)\) i.e.\ \(M^{ev}(G)\geq \frac{(M_1(G)-3t(G))^2}{m}\). Equality follows from the equality condition of the Cauchy--Schwarz inequality.
    \end{enumerate}
\end{proof}
\begin{example} The bounds in Theorem \ref{thm4} are checked for the following graphs.
\begin{itemize}
    \item In the complete split graph \(CS(a,b)\), the closed neighbourhoods of the end-points of each edge are nested, and the end-points of each edge form a dominating set. Hence both the bounds of (a) are attained and \(M^{ev}(CS(a,b))=(a+b)^2\left(\binom{a}{2}+ab\right)\). For instance, \(M^{ev}(CS(3,2))=25\times 9=225\).
    \item Let \(G=2K_3\cup K_1\). Every component is complete, so equality holds in (b): \(M^{ev}(G)=\frac{1
    }{4}HM(G)+M_1(G)+m=24+24+6=54\). The lower bound in (c) is \(6\times (1+\frac{12}{7})^2\approx 44.204\); it is not attained because of the isolated vertex.
    \item The cube-graph \(Q_3\) is \(ev\)-regular, with \(ev\)-degree equal to \(6\). Hence by (d), \(M^{ev}(Q_3)=\frac{72^2}{12}=432\).
    \item For the graph \(P_5\), (a) gives \(36< M^{ev}(P_5)\leq 50\). The upper bound is attained because \(P_5\) is a triangle-free graph.
\end{itemize}
\end{example}
The next theorem was first stated by Ediz~\cite{ediz2017predicting}; we give short proofs based on Theorem~\ref{thm3}.
\begin{theorem}[Extremal graphs]\label{thm5} Let \(n\geq 3\).
\begin{enumerate}[(a)]
    \item For every graph \(G\) of order \(n\), \(M^{ev}(G)\leq \frac{1}{2}n^3(n-1)\), with equality if and only if \(G\cong K_n\).
    \item For every tree \(T\) of order \(n\), \(16n-30\leq M^{ev}(T)\leq n^2(n-1)\). Equality holds on the left if and only if \(T\cong P_n\) and equality on the right holds if and only if \(T\cong K_{1,n-1}\).
    \item For every connected graph \(G\) of order \(n\), \(16n-30\leq M^{ev}(G)\leq \frac{1}{2}n^3(n-1)\). Equality on the left holds if and only if \(G\cong P_n\) and equality on the right holds if and only if \(G\cong K_n\).
\end{enumerate}
\end{theorem}
\begin{proof}
    \begin{enumerate}[(a)]
        \item \(G\) is a spanning subgraph of \(K_n\) so, by Theorem \ref{thm3}, \(M^{ev}(G)\leq M^{ev}(K_n)=\frac{1}{2}n^3(n-1)\), with equality if and only if \(G=K_n\).
        \item A tree is triangle-free, so \(M^{ev}(T)=HM(T)\) by Theorem~\ref{thm2}, and the claim follows from Theorem~\ref{thm:HMtrees}.
        \item Let \(T\) be a spanning tree of \(G\). By Theorem~\ref{thm3} and (b), \(M^{ev}(G)\geq M^{ev}(T)\geq 16n-30\). Equality forces \(G=T\) (Theorem~\ref{thm3}) and \(T\cong P_n\). The upper bound is (a).
    \end{enumerate}
\end{proof}
\begin{theorem}[Tur\'an-type bound]\label{thm6} Let \(r\geq 2\) and \(G\) be a \(K_{r+1}\)-free graph of order \(n\). Then \[M^{ev}(G)\leq n^2t_r(n),\] with equality if and only if \(G\cong T_r(n)\).
\end{theorem}
\begin{proof}
    By Theorem \ref{thm4}, part (a), and Theorem~\ref{thm:turan}, \(M^{ev}(G)\leq n^2m\leq n^2t_r(n)\). The end-points of each edge of \(T_r(n)\) form a dominating set, so \(M^{ev}(T_r(n))= n^2t_r(n)\) by Theorem \ref{thm4}, part (a). 
    
    Suppose equality holds, i.e.\ \(M^{ev}(G)= n^2m= n^2t_r(n)\). By the equality conditions of Theorem \ref{thm4}, part (a) and Theorem~\ref{thm:turan}, \(M^{ev}(G)= n^2m\) implies that the end-points of each edge of \(G\) form a dominating set, and \(n^2m= n^2t_r(n)\) implies \(G\cong T_r(n)\). Since \(T_r(n)\) satisfies the first condition (shown above), the two conditions together give \(G\cong T_r(n)\).
 Conversely, if \(G\cong T_r(n)\), then \(M^{ev}(T_r(n))=n^2t_r(n)\). Hence equality holds if and only if \(G\cong T_r(n)\). 
\end{proof}
\begin{corollary}\label{cor1}
    If \(G\) is a triangle-free graph (in particular bipartite), then \(M^{ev}(G)\leq n^2\lfloor\frac{n^2}{4}\rfloor\), with equality if and only if \(G\cong K_{\lfloor\frac{n}{2}\rfloor,\lceil\frac{n}{2}\rceil}\).
\end{corollary}
\begin{example}
    For \(n=7\), Corollary \ref{cor1} gives \(M^{ev}(G)\leq 49\times 12=588=M^{ev}(K_{3,4})\) for all triangle-free graphs of order \(7\), while Theorem \ref{thm6} with \(r=3\) gives \(M^{ev}(G)\leq 49\times 16=784=M^{ev}(K_{2,2,3})\) for \(K_4\)-free graphs of order \(7\).
\end{example}
\begin{theorem}[Nordhaus--Gaddum type inequality]\label{thm7}
    For every graph \(G\) of order \(n\geq 2\),
    \[\frac{n(n-1)(n+1)^2}{8}< M^{ev}(G)+M^{ev}(\overline{G})\leq \frac{n^3(n-1)}{2}.
    \]
    The upper bound is attained if and only if \(G\cong K_n\) or \(G\cong \overline{K_n}\).
\end{theorem}
\begin{proof}
    \textit{Upper Bound.} By Theorem \ref{thm4}, part (a), \(M^{ev}(G)+M^{ev}(\overline{G})\leq n^2(m+\overline{m})=n^2\binom{n}{2}=\frac{n^3(n-1)}{2}\). Equality holds if and only if every edge of \(G\) \(ev\)-dominates \(G\) and every edge of \(\overline{G}\) \(ev\)-dominates \(\overline{G}\). \(K_n\) and \(\overline{K_n}\) satisfy this. Suppose \(G\) is neither empty nor complete and equality holds. Let \(ab\) be an edge in \(G\) and \(uv\) be an edge in \(\overline{G}\). If \(u\) is an isolated vertex, then the edge \(ab\) cannot \(ev\)-dominate \(u\), a contradiction. So let \(w\) be a neighbour of \(u\). The vertex \(v\) is \(ev\)-dominated by the edge \(uw\), which implies \(v\in N_G(w)\) (\(v\notin N_G(u)\)). Therefore, \(w\) is a common neighbour of both \(u\) and \(v\), which implies that the edge \(uv\) cannot \(ev\)-dominate \(w\) in \(\overline{G}\), which is a contradiction. Hence equality in the upper bound holds if and only if \(G\cong K_n\) or \(G\cong \overline{K_n}\). \\
    \textit{Lower Bound.} Put \(f(x)=x\left(1+\frac{2x}{n}\right)^2\), which is a strictly convex function on \([0,\infty)\), and let \(N=\binom{n}{2}\). By Theorem \ref{thm4}, part (c) (trivially valid when the graph has no edges), and Jensen's inequality, \(M^{ev}(G)+M^{ev}(\overline{G})\geq f(m)+f(N-m)\geq 2f(\frac{N}{2})=N\left(1+\frac{n-1}{2}\right)^2=\frac{1}{8}n(n-1)(n+1)^2\). Equality from Theorem \ref{thm4}, part (c), and Jensen's inequality, would require both \(G\) and \(\overline{G}\) to be edgeless or of the form \(kK_r (r\geq2)\), and \(m=\overline{m}\). If \(G=kK_r\) with \(k,r\geq 2\), then \(\overline{G}\) is a connected, complete multipartite graph that is not complete, a contradiction; if \(G\) is edgeless or \(G=K_n\) (i.e.\ \(k=1\)), then \(m\neq \overline{m}\). Hence the inequality is strict.
\end{proof}
\begin{remark}
    The lower bound in Theorem~\ref{thm7} is not sharp. An exhaustive computer search over all graphs of order \(n\leq 7\) gives the following minimum values of \(M^{ev}(G)+M^{ev}(\overline{G})\).
    \begin{center}
        \begin{tabular}{@{}ccccccc@{}}
        \toprule
   \(n\) & 2 & 3 & 4 & 5 & 6 & 7\\
   \midrule
minimum & 4 & 22 & 68 & 160 & 348 & 646\\
bound: \(\frac{1}{8}n(n-1)(n+1)^2\) & 2.25 & 12 & 37.5 & 90 & 183.75 & 336\\
\bottomrule
        \end{tabular}
    \end{center}
    For \(n=4\) and \(n=5\) the minimum is attained by the self-complementary graphs \(P_4\) and \(C_5\). Determining the exact minimum for all \(n\) remains open.
\end{remark}
\subsection{Complement of a graph}
\begin{lemma}\label{l3}
    If \(u\neq v\) are non-adjacent in \(G\), then \(\deg^{ev}_{\overline{G}}(uv)=n-c_G(u,v)\).
\end{lemma}
\begin{proof}
    By Lemma \ref{l2}, \(\deg^{ev}_{\overline{G}}(uv)=n-|V(G)\setminus \left(N_{\overline{G}}[u]\cup N_{\overline{G}}[v]\right)|\). A vertex lies outside \(N_{\overline{G}}[u]\) if and only if it lies in \(N_{G}(u)\). So the vertices outside \(N_{\overline{G}}[u]\cup N_{\overline{G}}[v]\) form \(N_G(u)\cap N_G(v)\). Hence \(|V(G)\setminus \left(N_{\overline{G}}[u]\cup N_{\overline{G}}[v]\right)|=c_G(u,v)\), and the claim follows.
\end{proof}
\begin{theorem}\label{thm8}  
    For every graph \(G\) of order \(n\),
    \[M^{ev}(\overline{G})=n^2\overline{m}+\frac{M_1(G)}{2}-m+4q(G)-T_2(G)-2n \left[\frac{M_1(G)}{2}-m- 3t(G)\right].\]
\end{theorem}
\begin{proof}
    By Lemma \ref{l3}, \(M^{ev}(\overline{G})=\displaystyle\sum_{uv\in E(\overline{G})}\bigl[n-c_G(u,v)\bigr]^2\). Note that \(c_G(u,v)^2=c_G(u,v)+2\binom{c_G(u,v)}{2}\). Now 
    \begin{align*}
        M^{ev}(\overline{G})&=\sum_{uv\in E(\overline{G})}\bigl[n^2-2nc_G(u,v)+c_G(u,v)^2\bigr]\\
        &=\sum_{uv\in E(\overline{G})}n^2+\sum_{uv\in E(\overline{G})}c_G(u,v)^2-2n \sum_{uv\in E(\overline{G})}c_G(u,v)\\
        &=n^2\overline{m}+\sum_{uv\in E(\overline{G})}c_G(u,v)^2-2n \sum_{uv\in E(\overline{G})}c_G(u,v)
    \end{align*}
    \begin{align*}
  \sum_{uv\in E(\overline{G})}c_G(u,v)^2&=\sum_{\{u,v\}}c_G(u,v)^2-\sum_{uv\in E(G)}c_G(u,v)^2   \\
  &=\sum_{\{u,v\}}\left[c_G(u,v)+2\binom{c_G(u,v)}{2}\right]- \sum_{uv\in E(G)}t_G(uv)^2 \\
  &=\sum_{\{u,v\}}c_G(u,v)+ 2\sum_{\{u,v\}}\binom{c_G(u,v)}{2}- \sum_{uv\in E(G)}t_G(uv)^2\\
  &=\frac{M_1(G)}{2}-m+4q(G)-T_2(G)
    \end{align*}
    \begin{align*}
       \sum_{uv\in E(\overline{G})}c_G(u,v)&=\sum_{\{u,v\}}c_G(u,v)-\sum_{uv\in E(G)}c_G(u,v) \\
      &= \sum_{\{u,v\}}c_G(u,v)-\sum_{uv\in E(G)}t_G(uv) \\
       &=\frac{M_1(G)}{2}-m- 3t(G)
    \end{align*}
    Here we used Lemma~\ref{lem:counting}(b) for \(\displaystyle\sum_{\{u,v\}}c_G(u,v)\), Lemma~\ref{lem:counting}(c) for \(\displaystyle\sum_{\{u,v\}}\binom{c_G(u,v)}{2}\), and 
    Lemma~\ref{lem:counting}(a) for \(\displaystyle\sum_{uv\in E(G)}t_G(uv)=3t(G)\).
    Substituting these two sums into the expression for \(M^{ev}(\overline{G})\) proves the theorem.
\end{proof}
\begin{corollary}\label{cor2}
    If \(G\) has girth at least \(5\), then \(M^{ev}(\overline{G})=n^2\overline{m}+\frac{M_1(G)}{2}-m-n \left[M_1(G)-2m\right]\). Moreover, if \(G\) is \(r\)-regular, then \(M^{ev}(\overline{G})=\frac{n^2}{2}\left[n(n-1-r)-2r(r-1)\right]+\frac{nr(r-1)}{2}\). 
\end{corollary}
\begin{example}
    \begin{enumerate}[(i)]
        \item \(P_4\) is self-complementary. Theorem \ref{thm8} together with \(n=4,m=3, M_1(P_4)=10, t(P_4)=q(P_4)=T_2(P_4)=0\) gives \(M^{ev}(\overline{P_4})=48+2-16=34=M^{ev}(P_4)\). Similarly, \(M^{ev}(\overline{C_5})=M^{ev}(C_5)=125+5-50=80\).
        \item By Corollary \ref{cor2}, \(M^{ev}(\overline{C_n})=\frac{1}{2}n^3(n-3)+n-2n^2\) for \(n\geq 5\), and \(M^{ev}(\overline{P_n})=\frac{1}{2}n^4-\frac{3}{2}n^3-n^2+5n-2\) for \(n\geq 2\).
        \item The Petersen graph \(P\) has girth \(5\), \(n=10 , r=3\). By Corollary \ref{cor2}, \(M^{ev}(\overline{P})=50\times 48+ 30=2430\).
        \item \(\overline{K_{a,b}}=K_a\cup K_b\). Here \(t(K_{a,b})=0\), \(q(K_{a,b})=\binom{a}{2}\binom{b}{2}\), \(M_1(K_{a,b})=ab(a+b)\). By Theorem \ref{thm8}, \(M^{ev}(\overline{K_{a,b}})=\frac{1}{2}a^3(a-1)+\frac{1}{2}b^3(b-1)\).
        \item \(4\)-\textit{cycles matter}: \(C_8\) and \(2C_4\) are \(2\)-regular, triangle-free, have \(8\) vertices and the same ev-Zagreb index \(128\); but \(q(C_8)=0\) and \(q(2C_4)=2\) and accordingly \(M^{ev}(\overline{C_8})=1160\) while \(M^{ev}(\overline{2C_4})=1168\). So \(M^{ev}(\overline{G})\) is not determined by degree and triangle data of \(G\) alone.
    \end{enumerate}
\end{example}
\subsection{Line graph of a graph}
Recall that the edges of \(L(G)\) correspond to paths \(awb\) of length two in \(G\) (the edge \(\{wa,wb\}\) of \(L(G)\)), so \(m(L(G))=\displaystyle\sum_{w\in V(G)}\binom{d_G(w)}{2}\). We first record the ev-degree in \(L(G)\).
\begin{lemma}\label{l4}  
    Let \(e=wa\) and \(f=wb\) be distinct edges of \(G\) with common vertex \(w\). Then 
    \[\deg^{ev}_{L(G)}(ef)=d_G(w)+d_G(a)+d_G(b)-2-[ab\in E(G)].\] 
\end{lemma}
\begin{proof}
    \(d_{L(G)}(e)=d_G(w)+d_G(a)-2\) and \(d_{L(G)}(f)=d_G(w)+d_G(b)-2\). An edge of \(G\) adjacent to both \(e\) and \(f\) is either another edge at \(w\) (\(d_G(w)-2\) of them ) or the edge \(ab\), if present. So \(t_{L(G)}(ef)=d_G(w)-2+[ab\in E(G)]\) and Lemma~\ref{l2} gives the formula.
\end{proof}
\begin{proposition}\label{p2}
    \(m(L(G))=\frac{1}{2}M_1(G)-m(G)\), \(M_1(L(G))=HM(G)-4M_1(G)+4m(G)\), \(t(L(G))=t(G)+\displaystyle\sum_{u\in V(G)}\binom{d_G(u)}{3}\), and \(\displaystyle\sum_{f\in E(L(G))}\deg^{ev}_{L(G)}(f)=M_1(L(G))-3t(L(G))\).
\end{proposition}
\begin{proof}
    The first identity was noted above; the second one is \(\displaystyle\sum_{uv\in E(G)}\bigl(d_G(u)+d_G(v)-2\bigr)^2\). A triangle of \(L(G)\) either comes from a triangle of \(G\) or from three edges at a common vertex, and these two cases are disjoint, the second contributing \(\displaystyle\sum_{u\in V(G)}\binom{d_G(u)}{3}\). The last identity is Proposition~\ref{p1} for \(L(G)\).
\end{proof}
The main result of this section expresses \(M^{ev}(L(G))\) in terms of known invariants of \(G\).
\begin{theorem}\label{thm9}  
    For every graph \(G\),
    \[M^{ev}(L(G))=\frac{1}{2}\bigl[Z_4(G)-9F(G)\bigr]+ 8M_1(G)-4m(G)+ 3\Pi(G)-12M_2(G)+NM_1(G)-3T_1(G)+15t(G).\]
\end{theorem}
\begin{proof}
    By Lemma~\ref{l4}, \(M^{ev}(L(G))=\displaystyle\sum_{w\in V(G)}\displaystyle\sum_{\{a,b\}\subseteq N_G(w)}\left(X_{awb}-[ab\in E(G)]\right)^2\) with \(X_{awb}=d_G(w)+d_G(a)+d_G(b)-2\). Expand \(\left(X_{awb}-[ab\in E(G)]\right)^2\) to get \(X_{awb}^2-2X_{awb}[ab\in E(G)]+[ab\in E(G)]\).

    Now \begin{equation}\label{eq1}
   \sum_{w\in V(G)}\sum_{\{a,b\}\subseteq N_G(w)}[ab\in E(G)]=3t(G),
    \end{equation} since each triangle \(\{a,w,b\}\) of \(G\) is counted once for each choice of its middle vertex.

    For a triangle \(\{a,w,b\}\), the value \(X_{awb}=d_G(w)+d_G(a)+d_G(b)-2\) does not depend on which vertex is the middle one, so \(\displaystyle\sum_{w\in V(G)}\displaystyle\sum_{\{a,b\}\subseteq N_G(w)}X_{awb}[ab\in E(G)]=3\times \displaystyle\sum_{\substack{\text{all triangles}\\ \text {in }G}}(d_G(w)+d_G(a)+d_G(b)-2)\). 
    Since \(\displaystyle\sum_{\text{triangles}}\bigl(d_G(w)+d_G(a)+d_G(b)\bigr)=\displaystyle\sum_{u\in V(G)}t_G(u)d_G(u)=\frac{1}{2}T_1(G)\) by Lemma~\ref{lem:degreeids}(d), we get
    \begin{equation}\label{eq2}
    \sum_{w\in V(G)}\sum_{\{a,b\}\subseteq N_G(w)}X_{awb}[ab\in E(G)]=\frac{3}{2}T_1(G)-6t(G).\end{equation}

    For the term \(\displaystyle\sum_{w\in V(G)}\displaystyle\sum_{\{a,b\}\subseteq N_G(w)}X_{awb}^2\), fix \(w\), write \(d=d_G(w)\), \(S=S_w=\displaystyle\sum_{a\in N_G(w)}d_G(a)\) and \(Q=Q_w=\displaystyle\sum_{a\in N_G(w)}d_G(a)^2\). Over the \(\binom{d}{2}\) pairs \(\{a,b\}\subseteq N_G(w)\),
    \begin{align*}
    \sum_{\{a,b\}\subseteq N_G(w)}(d_G(a)+d_G(b))&=(d-1)S,\\ 
    \sum_{\{a,b\}\subseteq N_G(w)}(d_G(a)+d_G(b))^2&=(d-1)Q+(S^2-Q)=(d-2)Q+S^2, 
    \end{align*}
so \(\displaystyle\sum_{\{a,b\}\subseteq N_G(w)}X_{awb}^2=\binom{d}{2}(d-2)^2+2(d-2)(d-1)S+(d-2)Q+S^2\). Now sum over \(w\), using \(\displaystyle\sum_{w\in V(G)}\binom{d_G(w)}{2}(d_G(w)-2)^2=\frac{1}{2}\left(Z_4(G)-5F(G)+8M_1(G)-8m(G)\right)\), \(\displaystyle\sum_{w\in V(G)}(d_G(w)^2-3d_G(w)+2)S_w=\Pi(G)-6M_2(G)+2M_1(G)\), \(\displaystyle\sum_{w\in V(G)}(d_G(w)-2)Q_w=\Pi(G)-2F(G)\) (by Lemma~\ref{lem:degreeids}(b), (c)) and \(\displaystyle\sum_{w\in V(G)}S_w^2=NM_1(G)\):
\begin{equation}\label{eq3}
\sum_{w\in V(G)}\sum_{\{a,b\}\subseteq N_G(w)}X_{awb}^2=\frac{1}{2}Z_4(G)-\frac{9}{2}F(G)+8M_1(G)-4m(G)+ 3\Pi(G)-12M_2(G)+NM_1(G).
\end{equation}
Combining Equations \eqref{eq1}, \eqref{eq2} and \eqref{eq3}, we get 
\begin{align*}
M^{ev}(L(G))&=\sum_{w\in V(G)}\sum_{\{a,b\}\subseteq N_G(w)} X_{awb}^2-2\Bigl(\tfrac32T_1(G)-6t(G)\Bigr)+3t(G)
  \\
  &=\sum_{w\in V(G)}\sum_{\{a,b\}\subseteq N_G(w)} X_{awb}^2-3T_1(G)+ 15t(G),
\end{align*}
which, by Equation \eqref{eq3}, is the stated formula.
\end{proof}
\begin{corollary}\label{cor3}
    If \(G\) is triangle-free (for example a tree), then \(M^{ev}(L(G))=\frac{1}{2}\bigl[Z_4(G)-9F(G)\bigr]+ 8M_1(G)-4m(G)+ 3\Pi(G)-12M_2(G)+NM_1(G).\)
\end{corollary}
\begin{corollary}\label{cor4}
    If \(G\) is \(r\)-regular of order \(n\), then 
    \[M^{ev}(L(G))=\frac{nr(r-1)(3r-2)^2}{2}-3(6r-5)t(G).\]
\end{corollary}
\begin{proof}
    For an \(r\)-regular graph, \(m(G)=\frac{nr}{2}, M_1(G)=nr^2, F(G)=nr^3, Z_4(G)=\Pi(G)=NM_1(G)=nr^4, M_2(G)=\frac{nr^3}{2}\) and \(T_1(G)=6rt(G)\). Substituting into the formula from Theorem~\ref{thm9}, 
    \begin{align*}
        M^{ev}(L(G))&=\frac{nr}{2}(9r^3-21r^2+16r-4)-3(6r-5)t(G)\\
        &=\frac{nr}{2}(r-1)(3r-2)^2-3(6r-5)t(G).
    \end{align*}
\end{proof}
\begin{corollary}\label{cor5}
    \(M^{ev}(L(G))\leq HM(L(G))\), with equality if and only if every component of \(G\) is a path or a cycle of length at least \(4\).
\end{corollary}
\begin{proof}
    By Theorem~\ref{thm2}, equality holds if and only if \(L(G)\) is triangle-free, i.e. if and only if \(G\) has no vertex of degree \(\geq 3\) and no triangle.
\end{proof}
\begin{example}[Line graphs of standard graphs]
\begin{enumerate}[(i)]
    \item \(L(P_n)=P_{n-1}\) and \(L(C_n)=C_n\). For \(P_n\), we have 
  \(m=n-1, M_1=4n-6, F=8n-14, Z_4=16n-30, M_2=4n-8, \Pi=16n-36\) and \(NM_1=16n-38\). Hence by Corollary~\ref{cor3}, \(M^{ev}(L(P_n))=16n-46\) for \(n\geq 4\) and \(M^{ev}(L(P_3))=4\). By Corollary~\ref{cor4}, \(M^{ev}(L(C_n))=16n\) for \(n\geq 4\) and \(M^{ev}(L(C_3))=48-21=27\).
  \item \(L(K_{1,n-1})=K_{n-1}\). By Corollary~\ref{cor3}, \(M^{ev}(L(K_{1,n-1}))=\frac{1}{2}(n-1)^3(n-2)\) for \(n\geq 2\).
  \item \(L(K_n)\) is the triangular graph. Corollary~\ref{cor4} with \(r=n-1\) and \(t=\binom{n}{3}\) gives \(M^{ev}(L(K_{n}))=\frac{9}{2}n(n-1)(n-2)^3\).
  \item The Petersen graph \(P\) and the cube graph \(Q_3\) are \(3\)-regular and triangle-free: \(M^{ev}(L(P))=1470\) and \(M^{ev}(L(Q_3))=1176\). 
  \item \(L(K_{a,b})=K_a\square K_b\) is the rook's graph. In \(K_{a,b}\), a path \(a'wb'\) with \(w\) in the part of size \(a\) has its ends in the part of size \(b\), and Lemma~\ref{l4} gives ev-degree equal to \(b+2a-2\). Hence 
  \[M^{ev}(L(K_{a,b}))=a\binom{b}{2}(b+2a-2)^2+b\binom{a}{2}(a+2b-2)^2.\]
  \item Let \(G=W_n\) be the wheel graph with \(n\geq 4\). Then \(m=2n, M_1=n^2+9n, M_2=3n^2+9n, F=n^3+27n, Z_4=n^4+81n, \Pi=3n^3+9n^2+54n, NM_1=n^3+21n^2+36n, t=n, T_1=2n^2+12n\). By Theorem~\ref{thm9}, \(M^{ev}(L(W_n))=\frac{1}{2}n^4+\frac{11}{2}n^3+14n^2+52n\) for \(n\geq 4\).
\end{enumerate}
\end{example}

\subsection{Subdivision graph \(S(G)\)}
\begin{theorem}
    \(M^{ev}(S(G))= F(G)+ 4M_1(G)+ 8m(G)\)
\end{theorem}
\begin{proof}
    Since \(S(G)\) is triangle-free, \(M^{ev}(S(G))=HM(S(G))=F(S(G))+ 2M_2(S(G))\). Observe that \(F(S(G))=F(G)+ 8m(G)\) and \(M_2(S(G))=2M_1(G)\). Hence the result.
\end{proof}
\subsection{Edge-semitotal graph \(Q(G)\)}
\begin{theorem}
    \(M^{ev}(Q(G))= M^{ev}(L(G))+ 8M_1(L(G))+ 33m(L(G))+ 18m(G)- 18t(L(G))\)
\end{theorem}
\begin{proof}
Write \(E(Q(G))=E_1\cup E_2\), where \(E_1\) is the set of edges joining two new vertices (the edges of \(L(G)\)) and \(E_2\) the set of edges joining an old vertex to a new vertex.
     We observe that, if \(x\in V(Q(G)) \) then \(d_{Q(G)}(x)=\left\{
                              \begin{array}{ll}
                               d_{G}(x) & \hbox{ if \(x\in V(G)\),}\\
d_{L(G)}(x)+ 2 & \hbox{ if \(x\in E(G)\).}\\
                              \end{array}
                            \right.\) This implies that

\begin{eqnarray*}
  t_{Q(G)}(xy) &=& \left\{
                              \begin{array}{ll}
                               t_{L(G)}(xy)+1 & \hbox{ if \(e=xy\in E_1\),}\\
d_G(x)-1 & \hbox{ if \(e=xy\in E_2\).}\\
                              \end{array}
                            \right.\\
\implies \deg^{ev}_{Q(G)}(xy)&=&\left\{
                              \begin{array}{ll}
                              \deg^{ev}_{L(G)}(xy)+ 3  & \hbox{ if \(e=xy\in E_1\),}\\
d_{L(G)}(y)+ 3 & \hbox{ if \(e=xy\in E_2\).}\\
                              \end{array}
                            \right.\\
\end{eqnarray*}
    Now,
    \begin{eqnarray*}
S_1&=&M^{ev}(L(G))+ 6M_1(L(G))- 18t(L(G))+ 9m(L(G)),
\\
S_2&=&\displaystyle\sum_{x\in V(G)}\displaystyle\sum_{\substack{y \text{ is an edge} \\ \text{ incident to }x}}\Bigl(d_{L(G)}(y)+ 3\Bigr)^2= 2M_1(L(G))+ 24m(L(G))+ 18m(G).
 \end{eqnarray*}
Since \(M^{ev}(Q(G))=S_1+S_2\), the result follows.
\end{proof}
\subsection{Vertex-semitotal graph \(R(G)\)}
\begin{theorem}
    \(M^{ev}(R(G))=M^{ev}(G)+ 3HM(G)+ 4F(G)- 2T_1(G)+ 6t(G)+ 3m(G)\)
\end{theorem}
\begin{proof}
Write \(E(R(G))=E_1\cup E_2\), where \(E_1=E(G)\) (old--old edges) and \(E_2\) is the set of old--new edges.
   We observe that, if \(x\in V(R(G)) \) then \(d_{R(G)}(x)=\left\{
                              \begin{array}{ll}
                               2d_{G}(x) & \hbox{ if \(x\in V(G)\),}\\
2 & \hbox{ if \(x\in E(G)\).}\\
                              \end{array}
                            \right.\) This implies that

\begin{eqnarray*}
  t_{R(G)}(xy) &=& \left\{
                              \begin{array}{ll}
                               t_G(xy)+1 & \hbox{ if \(e=xy\in E_1\),}\\
1 & \hbox{ if \(e=xy\in E_2\).}\\
                              \end{array}
                            \right.\\
\implies \deg^{ev}_{R(G)}(xy)&=&\left\{
                              \begin{array}{ll}
                              2(d_G(x)+ d_G(y))-t_G(xy)-1  & \hbox{ if \(e=xy\in E_1\),}\\
2d_G(x)+1 & \hbox{ if \(e=xy\in E_2\).}\\
                              \end{array}
                            \right.\\
\end{eqnarray*}
    Now,
    \begin{eqnarray*}
S_1&=&\displaystyle\sum_{e\in E_1}\deg^{ev}_{R(G)}(e)^2=
3HM(G)+ M^{ev}(G)- 4M_1(G)+ m(G)+ 6t(G)- 2T_1(G),
\\
S_2&=&\displaystyle\sum_{e\in E_2}\deg^{ev}_{R(G)}(e)^2=\displaystyle\sum_{x\in V(G)}\displaystyle\sum_{\substack{y \text{ is an edge} \\ \text{ incident to }x}}\Bigl(2d_G(x)+ 1\Bigr)^2= \displaystyle\sum_{x\in V(G)}d_G(x)\Bigl(1+ 4d_G(x)+ 4d_G(x)^2\Bigr)\\ &=&4F(G)+ 4M_1(G)+ 2m(G).
 \end{eqnarray*}
 Since, \(M^{ev}(R(G))=S_1+S_2\), the result follows.
\end{proof}
\begin{corollary}
    For a triangle-free graph \(G\) we have \(M^{ev}(R(G))=4HM(G)+ 4F(G)+ 3m(G)\).
\end{corollary}
\subsection{Total graph \(T(G)\)}
\begin{theorem}
  \(M^{ev}(T(G))=M^{ev}(G)+ 3HM(G)- 4M_1(G)- 2T_1(G)+ 6t(G)+ m(G)+ M^{ev}(L(G))+ 6M_1(L(G))- 18t(L(G))+ 9m(L(G))+ 5F(G)+ 8M_2(G)\)
\end{theorem}
\begin{proof}
Write \(E(T(G))=E_1\cup E_2\cup E_3\), where \(E_1=E(G)\), \(E_2=E(L(G))\) and \(E_3\) is the set of old--new edges.
  Note that if \(x\in V(T(G)) \) then \(d_{T(G)}(x)=\left\{
                              \begin{array}{ll}
                               2d_{G}(x) & \hbox{ if \(x\in V(G)\),}\\
d_{L(G)}(x)+ 2 & \hbox{ if \(x\in E(G)\).}\\
                              \end{array}
                            \right.\) This implies that
\begin{eqnarray*}
  t_{T(G)}(xy) &=& \left\{
                              \begin{array}{ll}
                               t_G(xy)+1 & \hbox{ if \(e=xy\in E_1\),}\\
t_{L(G)}(xy)+ 1 & \hbox{ if \(e=xy\in E_2\),}\\
d_G(x) & \hbox{ if \(e=xy \in E_3\)}
                              \end{array}
                            \right.\\
\implies \deg^{ev}_{T(G)}(xy)&=&\left\{
                              \begin{array}{ll}
                              2(d_G(x)+ d_G(y))-t_G(xy)-1  & \hbox{ if \(e=xy\in E_1\),}\\
\deg^{ev}_{L(G)}(xy)+ 3 & \hbox{ if \(e=xy\in E_2\),}\\
2d_G(x)+ d_G(z) & \hbox{ if \(e=xy\in E_3 \),}\\
                              \end{array}                            
                            \right.\\
\end{eqnarray*}
where for \(e=xy\in E_3\) with \(x\in V(G)\), \(z\) denotes the other end-vertex of the edge \(y\) of \(G\).
    Now,
    \begin{eqnarray*}
  S_1&=&\displaystyle\sum_{e\in E_1}\deg^{ev}_{T(G)}(e)^2  =\displaystyle\sum_{xy\in E(G)}\Bigl(2(d_G(x)+ d_G(y))- t_G(xy)- 1\Bigr)^2
  \\& =&M^{ev}(G)+ 3HM(G)- 4M_1(G)- 2T_1(G)+ 6t(G)+ m(G),\\
  S_2&=&\displaystyle\sum_{e\in E_2}\deg^{ev}_{T(G)}(e)^2=\displaystyle\sum_{xy\in E(L(G))}\Bigl(\deg^{ev}_{L(G)}(xy)+ 3\Bigr)^2
 \\&=&M^{ev}(L(G))+ 6M_1(L(G))- 18t(L(G))+ 9m(L(G))
  ,\\
   S_3&=&\displaystyle\sum_{e\in E_3}\deg^{ev}_{T(G)}(e)^2=\displaystyle\sum_{x\in V(G)}\displaystyle\sum_{\substack{y \text{ is an edge}\\ \text{ incident to }x}}\deg^{ev}_{T(G)}(e)^2
  =\displaystyle\sum_{x\in V(G)}\displaystyle\sum_{z\in N_G(x)}\Bigl(2d_G(x)+ d_G(z)\Bigr)^2
 \\& =&5F(G)+ 8M_2(G).
 \end{eqnarray*}
 Since, \(M^{ev}(T(G))=S_1+S_2+S_3\), the result follows.
\end{proof}
\begin{corollary}
  If \(G\) is triangle-free then \(M^{ev}(T(G))=4HM(G)- 4M_1(G)+ M^{ev}(L(G))+ 6M_1(L(G))- 18t(L(G))+ 9m(L(G))+ 5F(G)+ 8M_2(G)+ m(G)\).
\end{corollary}

\subsection{Join of two graphs and its variants}
\begin{theorem}
     \(M^{ev}(G+H)= M^{ev}(G)+ M^{ev}(H)+ 2n(H)M_1(G)+ 2n(G)M_1(H)- 6n(H)t(G) -6n(G)t(H)+ n(H)^2m(G)+ n(G)^2m(H)+ \Bigl(n(G)+ n(H)\Bigr)^2n(G)n(H).\)
\end{theorem}
\begin{proof}
Write \(E(G+H)=E_1\cup E_2\cup E_3\), where \(E_1=E(G)\), \(E_2=E(H)\) and \(E_3=\{xy: x\in V(G),\ y\in V(H)\}\).
     We observe that, if \(x\in V(G+ H) \) then \(d_{G+ H}(x)=\left\{
                              \begin{array}{ll}
                               d_{G}(x)+ n(H) & \hbox{ if \(x\in V(G)\),}\\
d_H(x)+ n(G) & \hbox{ if \(x\in V(H)\).}\\
                              \end{array}
                            \right.\) This implies that

\begin{eqnarray*}
  t_{G+H}(xy) &=& \left\{
                              \begin{array}{ll}
                               t_G(xy)+ n(H)& \hbox{ if \(e=xy\in E_1\),}\\
t_H(xy)+ n(G) & \hbox{ if \(e=xy\in E_2\).}\\
d_G(x)+d_H(y) & \hbox{ if \(e=xy\in E_3\).}\\
                              \end{array}
                            \right.\\
\implies \deg^{ev}_{G+ H}(xy)&=&\left\{
                              \begin{array}{ll}
                              n(H)+ \deg^{ev}_G(xy)  & \hbox{ if \(e=xy\in E_1\),}\\
n(G)+ \deg^{ev}_H(xy) & \hbox{ if \(e=xy\in E_2\).}\\
n(G)+ n(H) & \hbox{ if \(e=xy\in E_3\).}\\
                              \end{array}
                            \right.\\
\end{eqnarray*}
    Now,
    \begin{eqnarray*}
S_1&=&M^{ev}(G)+ n(H)^2m(G)+ 2n(H)\Bigl(M_1(G)- 3t(G)\Bigr)
\\
S_2&=&M^{ev}(H)+ n(G)^2m(H)+ 2n(G)\Bigl(M_1(H)- 3t(H)\Bigr)\\
S_3&=&\Bigl(n(G)+ n(H)\Bigr)^2n(G)n(H)\\
 \end{eqnarray*}
 Since, \(M^{ev}(G+H)=S_1+S_2+S_3\), the result follows.
\end{proof}
The definition of the join extends to \(k\) graphs (Definition~\ref{def:binary}(i)). 
\begin{theorem}\label{thm:kjoin}
    Let \(G_1,\dots,G_k\) be pairwise vertex-disjoint graphs and \(N=\displaystyle\sum_{i=1}^{k}n(G_i)\). Then 
    \begin{align*}
    M^{ev}(G_1+\dots+ G_k)&=\sum_{i=1}^{k}\Bigl[M^{ev}(G_i)+m(G_i)\bigl(N-n(G_i)\bigr)^2+ 2\bigl(N-n(G_i)\bigr)\bigl(M_1(G_i)-3t(G_i)\bigr)\Bigr]\\
    &\quad+N^2\sum_{1\leq i<j\leq k}n(G_i)\,n(G_j).
    \end{align*}
\end{theorem}
\begin{proof}
    An edge \(e\) of \(G_i\) gains \(N-n(G_i)\) in the degree of each end-vertex and \(N-n(G_i)\) common neighbours, so by Lemma~\ref{l2}
    \(\deg^{ev}(e)=\deg^{ev}_{G_i}(e)+N-n(G_i)\). The end-vertices of an edge joining \(G_i\) and \(G_j\) form a dominating set, so its ev-degree is \(N\). Summing the squares and using Proposition~\ref{p1} gives the result.
\end{proof}
We now turn to the subdivision-based joins of Definition~\ref{def:subdivision-ops}(i)--(iii). The proofs are similar to the one above and are omitted.
\begin{theorem}
    \(M^{ev}(G\dot{\vee} H)=F(G)+ (3n(H)+4)M_1(G)+ M^{ev}(H)+ 2n(G)M_1(H)- 6n(G)t(H)+ 2m(G)(n(H)+2)^2+ n(G)^2m(H)+ 4(n(G)+ n(H))n(H)m(G)+ (n(G)+ n(H))^2n(G)n(H)\)
\end{theorem}
\begin{theorem}
    \(M^{ev}(G\veebar H)=F(G)+ 2(n(H)+2)M_1(G)+ M^{ev}(H)+ 2m(G)M_1(H)- 6m(G)t(H)+ 2m(G)(n(H)+2)^2+ m(G)^2m(H)+ (2+ m(G)+ n(H))^2m(G)n(H)\)
\end{theorem}
\begin{theorem}
    \(M^{ev}(G^{S}\triangleright (H^{V}\cup K^{E}))=F(G)+ \Bigl(3n(H)+ 2n(K) +4\Bigr)M_1(G)+ M^{ev}(H)+ 2n(G)M_1(H)- 6n(G)t(H)+ M^{ev}(K)+ 2m(G)M_1(K)- 6m(G)t(K)+ 2m(G)\Bigl(n(H)+ n(K) +2\Bigr)^2+ n(G)^2m(H)+ m(G)^2m(K)+ \Bigl(n(G)+ n(H)\Bigr)^2n(G)n(H)+ 4\Bigl(n(G)+ n(H)\Bigr)n(H)m(G)+\Bigl(2+ m(G)+ n(K)\Bigr)^2m(G)n(K)\)
\end{theorem}
\subsection{Cartesian product and its variants}
\begin{theorem}
    $M^{ev}(G\otimes H)= n(H)M^{ev}(G)+ n(G)M^{ev}(H)+ 12m(H)M_1(G)+ 12m(G)M_1(H)- 24m(G)t(H)- 24m(H)t(G).$
\end{theorem}
\begin{proof}
    The degree of a vertex \((x,y)\) in \(G\otimes H\) is \(d_G(x)+ d_H(y)\).\\
    Now 
    \begin{eqnarray*}
  t_{G\otimes H}(e) &=& \left\{
                              \begin{array}{ll}
                               t_{H}(yy') & \hbox{ if \(e=(x,y)(x,y')\in E_1\),}\\
t_G(xx') & \hbox{ if \(e=(x,y)(x',y)\in E_2\).}\\
                              \end{array}
                            \right.\\
\implies \deg^{ev}_{G\otimes H}(xy)&=&\left\{
                              \begin{array}{ll}
                             2d_G(x)+ \deg^{ev}_H(yy')  & \hbox{ if \(e=(x,y)(x,y')\in E_1\),}\\
2d_H(y)+ \deg^{ev}_G(xx') & \hbox{ if \(e=(x,y)(x',y)\in E_2\).}\\
                              \end{array}
                            \right.\\
\end{eqnarray*}
    Now,
    \begin{eqnarray*}
S_1&=&\displaystyle\sum_{x\in V(G)}\displaystyle\sum_{yy'\in E(H)}\Bigl(2d_G(x)+ \deg^{ev}_H(yy')\Bigr)^2\\
&=&n(G)M^{ev}(H)+ 4m(H)M_1(G)+ 8m(G)M_1(H)- 24m(G)t(H)
\\
S_2&=&n(H)M^{ev}(G)+ 4m(G)M_1(H)+ 8m(H)M_1(G)- 24m(H)t(G).
 \end{eqnarray*}
 Adding \(S_1\) and \(S_2\) gives the result.
\end{proof}
\begin{lemma}\label{lem:SQRTinv}
    For every graph \(G\) with \(m\) edges, \(S(G), Q(G), R(G), T(G)\) have \(n(G)+m\) vertices, and 
    \begin{enumerate}[(i)]
   \item     \(m(S(G))=2m\),  \(M_1(S(G))=M_1(G)+4m\),  \(t(S(G))=0\);
     \item  \(m(Q(G))=m(L(G))+2m\), \(M_1(Q(G))=M_1(G)+M_1(L(G))+8m(L(G))+4m\),  \(t(Q(G))=t(L(G))+m(L(G))\);
     \item  \(m(R(G))=3m\),  \(M_1(R(G))=4M_1(G)+4m\),  \(t(R(G))=t(G)+m\);
     \item  \(m(T(G))=m(L(G))+ 3m\),  \(M_1(T(G))=4M_1(G)+ M_1(L(G))+ 8m(L(G))+4m\),  \(t(T(G))=t(G)+t(L(G))+m(L(G))+m\).
    \end{enumerate}
\end{lemma}
\begin{proof}
    The degrees are: in \(S(G)\), \(d_G(x)\) for old and \(2\) for new vertices; in \(Q(G)\), \(d_G(x)\) and \(d_{L(G)}(e)+2\); in \(R(G)\), \(2d_G(x)\) and \(2\); in \(T(G)\), \(2d_G(x)\) and \(d_{L(G)}(e)+2\). Use \(\displaystyle\sum_{e\in E(G)}d_{L(G)}(e)=2m(L(G))\). A triangle of \(Q(G)\) is a triangle of \(L(G)\) or consists of two adjacent edges and their common vertex; a triangle of \(R(G)\) is a triangle of \(G\) or consists of an edge \(uv\) and \(u,v\); \(T(G)\) has all four kinds.
\end{proof}
\begin{corollary}\label{cor:cartSQRT}
    Let \(X\in \{S(G), Q(G), R(G), T(G)\}\). Then 
    \begin{multline*}
    M^{ev}(X\otimes H)=n(H)M^{ev}(X)+\bigl(n(G)+m(G)\bigr)M^{ev}(H)+12m(H)M_1(X)\\
    +12m(X)M_1(H)-24m(X)t(H)-24m(H)t(X),
    \end{multline*}
    with \(M^{ev}(X)\) from the theorems above and \(m(X),M_1(X),t(X)\) from Lemma~\ref{lem:SQRTinv}. For example, 
    \begin{align*}
    M^{ev}(S(G)\otimes H)&=n(H)\bigl[F(G)+4M_1(G)+8m(G)\bigr]+ \bigl[n(G)+m(G)\bigr]M^{ev}(H)\\
    &\quad+12m(H)\bigl[M_1(G)+4m(G)\bigr]+ 24m(G)M_1(H)-48m(G)t(H),\\
    M^{ev}(R(G)\otimes H)&=n(H)M^{ev}(R(G))+\bigl[n(G)+m(G)\bigr]M^{ev}(H)+48m(H)\bigl[M_1(G)+m(G)\bigr]\\
    &\quad +36m(G)M_1(H)-72m(G)t(H)-24m(H)\bigl[t(G)+m(G)\bigr].
    \end{align*}
\end{corollary}
\subsection{Composition operation}
\begin{theorem}
    \begin{align*}M^{ev}(G\circ H)&=n(H)^4M^{ev}(G)+ n(G)M^{ev}(H)+ n(H)^2m(H)M_1(G)+ 4n(H)m(G)M_1(H)\\& \quad- 12n(H)m(G)t(H).
    \end{align*}
\end{theorem}
\begin{proof}
    The degree of a vertex \((x,y)\) in \(G\circ H\) is \(n(H)d_G(x)+ d_H(y)\).\\
    Now 
    \begin{eqnarray*}
  t_{G\circ H}(e) &=& \left\{
                              \begin{array}{ll}
                               n(H)t_G(xx')+ d_H(y)+ d_H(y') & \hbox{ if \(e=xy\in E_1\),}\\
n(H)d_G(x)+ t_H(yy') & \hbox{ if \(e=xy\in E_2\).}\\
                              \end{array}
                            \right.\\
\implies \deg^{ev}_{G\circ H}(xy)&=&\left\{
                              \begin{array}{ll}
                            n(H)\deg^{ev}_G(xx')  & \hbox{ if \(e\in E_1\),}\\
n(H)d_G(x)+\deg^{ev}_H(yy') & \hbox{ if \(e\in E_2\).}\\
                              \end{array}
                            \right.\\
\end{eqnarray*}
    Now,
    \begin{eqnarray*}
S_1&=&\displaystyle\sum_{y\in V(H)}\displaystyle\sum_{y'\in V(H)}\displaystyle\sum_{xx'\in E(G)}\Bigl(n(H)\deg^{ev}_{G}(xx')\Bigr)^2
=n(H)^4M^{ev}(G)
\\
S_2&=&\displaystyle\sum_{x\in V(G)}\displaystyle\sum_{yy'\in E(H)}\Bigl(n(H)d_G(x)+ \deg^{ev}_{H}(yy')\Bigr)^2\\
&=&n(G)M^{ev}(H)+ n(H)^2m(H)M_1(G)+ 4n(H)m(G)M_1(H)- 12n(H)m(G)t(H).
 \end{eqnarray*}
 Adding \(S_1\) and \(S_2\) gives the result.
\end{proof}
\subsection{Corona product and its variants}
\begin{theorem}
    \begin{align*}M^{ev}(G\odot H)&= M^{ev}(G)+ n(G)M^{ev}(H)+ 5n(H)M_1(G)+ 2n(G)M_1(H)\\ &\quad + 8n(H)^2m(G)+ n(G)n(H)(n(H)+1)^2 + n(G)m(H)\\
    &\quad+ 4n(H)m(G)- 6n(G)t(H)- 12n(H)t(G).
    \end{align*}
\end{theorem}
\begin{proof}
Write \(E(G\odot H)=E_1\cup E_2\cup E_3\), where \(E_1=E(G)\), \(E_2\) is the set of edges inside the copies of \(H\), and \(E_3\) the set of edges joining a vertex of \(G\) to its copy of \(H\).
    The degree of a vertex \(x\) in $G\odot H$ is
$
    d_{G\odot H}(x)=\begin{cases}
        d_G(x)+ n(H) & x\in V(G)\\
        d_H(x)+ 1 & x\in V(\text{  copy of } H).
    \end{cases}
$\\
Now 
    \begin{eqnarray*}
  t_{G\odot H}(e) &=& \left\{
                              \begin{array}{ll}
                              t_G(e) & \hbox{ if \(e\in E_1\),}\\
t_H(e)+ 1 & \hbox{ if \(e\in E_2\).}\\
d_H(y) & \hbox{ if \(e\in E_3\).}\\
                              \end{array}
                            \right.\\
\implies \deg^{ev}_{G\odot H}(xy)&=&\left\{
                              \begin{array}{ll}
                             2n(H)+ \deg^{ev}_G(e)  & \hbox{ if \(e\in E_1\),}\\
 \deg^{ev}_H(e)+ 1 & \hbox{ if \(e\in E_2\).}\\
d_G(x)+ n(H)+ 1 & \hbox{ if \(e\in E_3\).}\\
                              \end{array}
                            \right.\\
\end{eqnarray*}

Now,
    \begin{eqnarray*}
S_1&=&M^{ev}(G)+ 4n(H)M_1(G)+ 4n(H)^2m(G)-12n(H)t(G)
\\
S_2&=&n(G)\Bigl[M^{ev}(H)+ 2M_1(H)- 6t(H)+ m(H)\Bigr]\\
S_3&=&n(H)\Bigl[M_1(G)+ n(G)\Bigl(n(H)+1\Bigr)^2+ 4\Bigl(n(H)+ 1\Bigr)m(G)\Bigr]
 \end{eqnarray*}
 Adding \(S_1\), \(S_2\) and \(S_3\) gives the result.
\end{proof}
The corona has several variants based on the subdivision graph (Definition~\ref{def:subdivision-ops}(iv)--(ix)). We now compute their ev-Zagreb indices.
\begin{theorem}
    \begin{align*}M^{ev}(G^{(S)}\odot H)&= F(G)+ \Bigl(3n(H)+ 4\Bigr)M_1(G)+ n(G)M^{ev}(H)+ 2n(G)M_1(H)- 6n(G)t(H) \\
    &\quad + 2m(G)\Bigl(n(H)+ 2\Bigr)^2 + n(G)m(H)\\& \quad + 4n(H)m(G)\Bigl(n(H)+ 1\Bigr)+ n(G)n(H)\Bigl(n(H)+ 1\Bigr)^2.
    \end{align*}
\end{theorem}
\begin{theorem}
    \begin{align*}M^{ev}(G^{(S)}\ominus H)&= F(G)+ 2\Bigl(n(H)+ 2\Bigr)M_1(G)+ m(G)M^{ev}(H)+ 2m(G)M_1(H)- 6m(G)t(H)\\
    & \quad + 2m(G)\Bigl(n(H)+ 2\Bigr)^2+ m(G)m(H)+ m(G)n(H)\Bigl( n(H)+ 3\Bigr)^2.\end{align*}
\end{theorem}

\begin{theorem}
    \begin{align*}M^{ev}(G^{(S)}\boxdot H)&=\Bigl(n(H)+ 1\Bigr)F(G)+ \Bigl[4(n(H)+ 1)^2+ m(H)\Bigr]M_1(G)+ n(G)M^{ev}(H)+ 4m(G)M_1(H)\\
    & \quad - 12m(G)t(H)+ 8m(G)\Bigl( n(H)+ 1\Bigr)^3.\end{align*}
\end{theorem}

\begin{theorem}   
    \begin{align*}M^{ev}(G^{(S)}\boxminus H)&=(1+ n(H))^3F(G)+ 4(1+ n(H))^2M_1(G)+ m(G)M^{ev}(H)\\ &\quad+ 4m(G)M_1(H)
    - 12m(G)t(H)
    + 8m(G)+ 4m(G)m(H)+ 8m(G)n(H) .
    \end{align*}
\end{theorem}

\begin{theorem}
\begin{align*}
    M^{ev}(G^{(S)}\circ (H, K))&=F(G)+ \Bigl(3n(H)+ 2n(K)+ 4\Bigr)M_1(G)+ n(G)M^{ev}(H)+ 2n(G)M_1(H)\\
    &\quad - 6n(G)t(H)
    + m(G)M^{ev}(K)+ 2m(G)M_1(K)- 6m(G)t(K)\\
    &\quad + 2\Bigl(n(H)+ n(K)+ 2\Bigr)^2m(G)+ n(G)m(H)+ m(G)m(K)\\
   &\quad +  4m(G)n(H)\Bigl(n(H)+ 1\Bigr)+ n(G)n(H)\Bigl(n(H)+ 1\Bigr)^2+ m(G)n(K)\Bigl(n(K)+ 3\Bigr)^2.
    \end{align*}
\end{theorem}

\begin{theorem}
    \begin{align*}M^{ev}(G^{S}\bullet(H, K))&=\Bigl[(1+n(K))^3+ n(H)\Bigr]F(G)\\ &\quad + \Bigl[4(1+ n(H))(1+ n(K))+ 4n(H)(1+ n(H))\\ & \qquad \quad + 4n(K)(1+ n(K))+ m(H)\Bigr]M_1(G)\\ & \quad + n(G)M^{ev}(H)+ 4m(G)M_1(H)- 12m(G)t(H)\\ & \quad + m(G)M^{ev}(K)+ 4m(G)M_1(K)- 12m(G)t(K)\\ & \quad + 8m(G)(1+ n(H))^3+ 4m(G)m(K)+ 8m(G)n(K).
    \end{align*}
\end{theorem}
\subsection{Tensor product}
\begin{theorem}
    If $G$ is a triangle-free graph, then $M^{ev}(G\times H)=F(G)F(H)+ 4M_2(G)M_2(H) .$
\end{theorem}
\begin{proof}
    From the definition, $E(G\times H)=\{(x,y)(x',y')\mid xx'\in E(G) \text{ and }yy'\in E(H)\}$. If $e\in E(G\times H),$ then $d_{G\times H}(x,y)=d_G(x)d_H(y)$, $d_{G\times H}(x',y')=d_G(x')d_H(y'),$ $t_{G\times H}(e)=0,$ and $\deg^{ev}_{G\times H}(e)=d_G(x)d_H(y)+ d_G(x')d_H(y').$ Thus, \\
$M^{ev}(G\times H)=\displaystyle\sum_{e\in E(G\times H)}\deg^{ev}_{G \times H}(e)^2
=\frac{1}{2}\displaystyle\sum_{x\in V(G)}\displaystyle\sum_{x'\in N_G(x)}\displaystyle\sum_{y\in V(H)}\displaystyle\sum_{y'\in N_H(y)}\Bigl(d_G(x)d_H(y)+ d_G(x')d_H(y')\Bigr)^2\\
~~~~~~~~~~~~~=\frac{1}{2}\displaystyle\sum_{x\in V(G)}\displaystyle\sum_{x'\in N_G(x)}\Bigl(d_G(x)^2F(H)+ d_G(x')^2F(H)+ 4d_G(x)d_G(x')M_2(H)\Bigr)
=F(G)F(H)+ 4M_2(G)M_2(H).
$
\end{proof}
\section{Conclusion}

In this chapter we studied the ev-Zagreb index of graphs. We showed that \(M^{ev}(G)=HM(G)-2T_1(G)+T_2(G)\), so the ev-Zagreb index is the hyper-Zagreb index corrected by the way the edges of \(G\) lie in triangles. In particular, the two indices coincide on triangle-free graphs, while on other graphs \(M^{ev}\) distinguishes graphs that every vertex-degree-based index treats as equal. Using this identity and the monotonicity of \(M^{ev}\) under taking subgraphs, we obtained sharp lower and upper bounds in terms of the order, the size, the Zagreb indices and the number of triangles, and characterised all extremal graphs. We also proved a Tur\'an-type bound for \(K_{r+1}\)-free graphs and a Nordhaus--Gaddum type inequality. We then derived exact formulas for the ev-Zagreb index of the complement and the line graph. These formulas involve the number of \(4\)-cycles and the neighbourhood degree sums of the original graph. Finally, we computed the index for the subdivision, semitotal and total graphs, for the join, Cartesian product, composition, corona, and tensor product, and for several subdivision-based joins and coronas. These formulas give the value of the ev-Zagreb index for whole families of graphs at once, and they show how the index depends on the triangle structure of the graphs involved.
\printbibliography

@article{chellali2017ve,
  author  = {Chellali, Mustapha and Haynes, Teresa W. and Hedetniemi, Stephen T. and Lewis, Thomas M.},
  title   = {On ve-degrees and ev-degrees in graphs},
  journal = {Discrete Mathematics},
  volume  = {340},
  number  = {2},
  pages   = {31--38},
  year    = {2017}
}

@article{gao2017sharp,
  title={Sharp Bounds of the Hyper-Zagreb Index on Acyclic, Unicylic, and Bicyclic Graphs},
  author={Gao, Wei and Jamil, Muhammad Kamran and Javed, Aisha and Farahani, Mohammad Reza and Wang, Shaohui and Liu, Jia-Bao},
  journal={Discrete Dynamics in Nature and Society},
  volume={2017},
  number={1},
  pages={Article ID 6079450},
  year={2017},
  publisher={Wiley Online Library}
}

@article{barik2018spectra,
  title={Spectra of graphs resulting from various graph operations and products: a survey},
  author={Barik, Sasmita and Kalita, Debajit and Pati, Sukanta and Sahoo, Gopinath},
  journal={Special Matrices},
  volume={6},
  number={1},
  pages={323--342},
  year={2018},
  publisher={De Gruyter Open}
}

@article{akhter2017computing,
  author  = {Akhter, Shehnaz and Imran, Muhammad},
  title   = {Computing the forgotten topological index of four operations on graphs},
  journal = {AKCE International Journal of Graphs and Combinatorics},
  volume  = {14},
  number  = {1},
  pages   = {70--79},
  year    = {2017}
}

@article{wiener1947structural,
  author  = {Wiener, Harry},
  title   = {Structural determination of paraffin boiling points},
  journal = {Journal of the American Chemical Society},
  volume  = {69},
  number  = {1},
  pages   = {17--20},
  year    = {1947}
}

@article{nikolic2003zagreb,
  author  = {Nikoli{\'c}, Sonja and Kova{\v{c}}evi{\'c}, Goran and Mili{\v{c}}evi{\'c}, Ante and Trinajsti{\'c}, Nenad},
  title   = {The {Z}agreb indices 30 years after},
  journal = {Croatica Chemica Acta},
  volume  = {76},
  number  = {2},
  pages   = {113--124},
  year    = {2003}
}

@article{gutman2004first,
  author  = {Gutman, Ivan and Das, Kinkar Ch.},
  title   = {The first {Z}agreb index 30 years after},
  journal = {MATCH Communications in Mathematical and in Computer Chemistry},
  volume  = {50},
  pages   = {83--92},
  year    = {2004}
}

@article{furtula2015forgotten,
  author  = {Furtula, Boris and Gutman, Ivan},
  title   = {A forgotten topological index},
  journal = {Journal of Mathematical Chemistry},
  volume  = {53},
  number  = {4},
  pages   = {1184--1190},
  year    = {2015}
}

@article{gutman2021geometric,
  author  = {Gutman, Ivan},
  title   = {Geometric approach to degree-based topological indices: {S}ombor indices},
  journal = {MATCH Communications in Mathematical and in Computer Chemistry},
  volume  = {86},
  number  = {1},
  pages   = {11--16},
  year    = {2021}
}

@phdthesis{peters1986theoretical,
  author = {Peters, Jr., Kenneth W.},
  title  = {Theoretical and Algorithmic Results on Domination and Connectivity},
  school = {Clemson University},
  year   = {1986}
}

@article{ediz2017predicting,
  author  = {Ediz, S{\"u}leyman},
  title   = {Predicting some physicochemical properties of octane isomers: a topological approach using ev-degree and ve-degree {Z}agreb indices},
  journal = {International Journal of Systems Science and Applied Mathematics},
  volume  = {2},
  number  = {5},
  pages   = {87--92},
  year    = {2017}
}

@article{shirdel2013hyper,
  author  = {Shirdel, Gholam Hassan and Rezapour, Hassan and Sayadi, A. M.},
  title   = {The hyper-{Z}agreb index of graph operations},
  journal = {Iranian Journal of Mathematical Chemistry},
  volume  = {4},
  number  = {2},
  pages   = {213--220},
  year    = {2013}
}

@article{hosoya1971topological,
  author  = {Hosoya, Haruo},
  title   = {Topological index. {A} newly proposed quantity characterizing the topological nature of structural isomers of saturated hydrocarbons},
  journal = {Bulletin of the Chemical Society of Japan},
  volume  = {44},
  number  = {9},
  pages   = {2332--2339},
  year    = {1971}
}

@article{gutman1972graph,
  author  = {Gutman, Ivan and Trinajsti{\'c}, Nenad},
  title   = {Graph theory and molecular orbitals. {T}otal $\varphi$-electron energy of alternant hydrocarbons},
  journal = {Chemical Physics Letters},
  volume  = {17},
  number  = {4},
  pages   = {535--538},
  year    = {1972}
}

@article{gutman1975graph,
  author  = {Gutman, Ivan and Ru{\v{s}}{\v{c}}i{\'c}, Branko and Trinajsti{\'c}, Nenad and Wilcox, Charles F.},
  title   = {Graph theory and molecular orbitals. {XII}. {A}cyclic polyenes},
  journal = {The Journal of Chemical Physics},
  volume  = {62},
  number  = {9},
  pages   = {3399--3405},
  year    = {1975}
}

@article{randic1975characterization,
  author  = {Randi{\'c}, Milan},
  title   = {Characterization of molecular branching},
  journal = {Journal of the American Chemical Society},
  volume  = {97},
  number  = {23},
  pages   = {6609--6615},
  year    = {1975}
}

@book{todeschini2000handbook,
  author    = {Todeschini, Roberto and Consonni, Viviana},
  title     = {Handbook of Molecular Descriptors},
  publisher = {Wiley-VCH},
  address   = {Weinheim},
  year      = {2000}
}

@article{gutman2013degree,
  author  = {Gutman, Ivan},
  title   = {Degree-based topological indices},
  journal = {Croatica Chemica Acta},
  volume  = {86},
  number  = {4},
  pages   = {351--361},
  year    = {2013}
}

@article{borovicanin2017bounds,
  author  = {Borovi{\'c}anin, Bojana and Das, Kinkar Ch. and Furtula, Boris and Gutman, Ivan},
  title   = {Bounds for {Z}agreb indices},
  journal = {MATCH Communications in Mathematical and in Computer Chemistry},
  volume  = {78},
  number  = {1},
  pages   = {17--100},
  year    = {2017}
}

@article{decaen1998upper,
  author  = {de Caen, Dominique},
  title   = {An upper bound on the sum of squares of degrees in a graph},
  journal = {Discrete Mathematics},
  volume  = {185},
  number  = {1--3},
  pages   = {245--248},
  year    = {1998}
}

@article{hansen2007comparing,
  author  = {Hansen, Pierre and Vuki{\v{c}}evi{\'c}, Damir},
  title   = {Comparing the {Z}agreb indices},
  journal = {Croatica Chemica Acta},
  volume  = {80},
  number  = {2},
  pages   = {165--168},
  year    = {2007}
}

@article{vukicevic2007comparing,
  author  = {Vuki{\v{c}}evi{\'c}, Damir and Graovac, Ante},
  title   = {Comparing {Z}agreb {$M_1$} and {$M_2$} indices for acyclic molecules},
  journal = {MATCH Communications in Mathematical and in Computer Chemistry},
  volume  = {57},
  number  = {3},
  pages   = {587--590},
  year    = {2007}
}

@article{bollobas1998graphs,
  author  = {Bollob{\'a}s, B{\'e}la and Erd{\H{o}}s, Paul},
  title   = {Graphs of extremal weights},
  journal = {Ars Combinatoria},
  volume  = {50},
  pages   = {225--233},
  year    = {1998}
}

@article{estrada1998atom,
  author  = {Estrada, Ernesto and Torres, Luis and Rodr{\'\i}guez, Lissette and Gutman, Ivan},
  title   = {An atom-bond connectivity index: modelling the enthalpy of formation of alkanes},
  journal = {Indian Journal of Chemistry},
  volume  = {37A},
  pages   = {849--855},
  year    = {1998}
}

@article{fajtlowicz1987conjectures,
  author  = {Fajtlowicz, Siemion},
  title   = {On conjectures of {G}raffiti-{II}},
  journal = {Congressus Numerantium},
  volume  = {60},
  pages   = {187--197},
  year    = {1987}
}

@article{vukicevic2009topological,
  author  = {Vuki{\v{c}}evi{\'c}, Damir and Furtula, Boris},
  title   = {Topological index based on the ratios of geometrical and arithmetical means of end-vertex degrees of edges},
  journal = {Journal of Mathematical Chemistry},
  volume  = {46},
  number  = {4},
  pages   = {1369--1376},
  year    = {2009}
}

@article{zhou2009novel,
  author  = {Zhou, Bo and Trinajsti{\'c}, Nenad},
  title   = {On a novel connectivity index},
  journal = {Journal of Mathematical Chemistry},
  volume  = {46},
  number  = {4},
  pages   = {1252--1270},
  year    = {2009}
}

@article{furtula2013structure,
  author  = {Furtula, Boris and Gutman, Ivan and Dehmer, Matthias},
  title   = {On structure-sensitivity of degree-based topological indices},
  journal = {Applied Mathematics and Computation},
  volume  = {219},
  number  = {17},
  pages   = {8973--8978},
  year    = {2013}
}

@article{milicevic2004reformulated,
  author  = {Mili{\v{c}}evi{\'c}, Ante and Nikoli{\'c}, Sonja and Trinajsti{\'c}, Nenad},
  title   = {On reformulated {Z}agreb indices},
  journal = {Molecular Diversity},
  volume  = {8},
  number  = {4},
  pages   = {393--399},
  year    = {2004}
}

@phdthesis{lewis2007vertex,
  author = {Lewis, Jason R.},
  title  = {Vertex-Edge and Edge-Vertex Parameters in Graphs},
  school = {Clemson University},
  year   = {2007}
}

@article{boutrig2016vertex,
  author  = {Boutrig, Razika and Chellali, Mustapha and Haynes, Teresa W. and Hedetniemi, Stephen T.},
  title   = {Vertex-edge domination in graphs},
  journal = {Aequationes Mathematicae},
  volume  = {90},
  number  = {2},
  pages   = {355--366},
  year    = {2016}
}

@article{horoldagva2019ve,
  author  = {Horoldagva, Batmend and Das, Kinkar Ch. and Selenge, Tsend-Ayush},
  title   = {On ve-degree and ev-degree of graphs},
  journal = {Discrete Optimization},
  volume  = {31},
  pages   = {1--7},
  year    = {2019}
}

@article{sahin2018ev,
  author  = {{\c{S}}ahin, B{\"u}nyamin and Ediz, S{\"u}leyman},
  title   = {On ev-degree and ve-degree topological indices},
  journal = {Iranian Journal of Mathematical Chemistry},
  volume  = {9},
  number  = {4},
  pages   = {263--277},
  year    = {2018}
}

@article{delen2022ve,
  author  = {Delen, Sadik and Khan, Riaz Hussain and Kamran, Muhammad and Salamat, Nadeem and Baig, Abdul Qudair and Cang{\"u}l, Ismail Naci and Pandit, M. K.},
  title   = {Ve-degree, ev-degree, and degree-based topological indices of fenofibrate},
  journal = {Journal of Mathematics},
  volume  = {2022},
  pages   = {Article ID 4477808},
  year    = {2022}
}

@article{khalifeh2009first,
  author  = {Khalifeh, M. H. and Yousefi-Azari, Hassan and Ashrafi, Ali Reza},
  title   = {The first and second {Z}agreb indices of some graph operations},
  journal = {Discrete Applied Mathematics},
  volume  = {157},
  number  = {4},
  pages   = {804--811},
  year    = {2009}
}

@article{de2016findex,
  author  = {De, Nilanjan and Nayeem, Sk. Md. Abu and Pal, Anita},
  title   = {{F}-index of some graph operations},
  journal = {Discrete Mathematics, Algorithms and Applications},
  volume  = {8},
  number  = {2},
  pages   = {1650025},
  year    = {2016}
}

@article{eliasi2009four,
  author  = {Eliasi, Mehdi and Taeri, Bijan},
  title   = {Four new sums of graphs and their {W}iener indices},
  journal = {Discrete Applied Mathematics},
  volume  = {157},
  number  = {4},
  pages   = {794--803},
  year    = {2009}
}

@article{deng2016zagreb,
  author  = {Deng, Hanyuan and Sarala, D. and Ayyaswamy, S. K. and Balachandran, Selvaraj},
  title   = {The {Z}agreb indices of four operations on graphs},
  journal = {Applied Mathematics and Computation},
  volume  = {275},
  pages   = {422--431},
  year    = {2016}
}

@article{liu2019reformulated,
  author  = {Liu, Jia-Bao and Ali, Bahadur and Malik, Muhammad Aslam and Siddiqui, Hafiz Muhammad Afzal and Imran, Muhammad},
  title   = {Reformulated {Z}agreb indices of some derived graphs},
  journal = {Mathematics},
  volume  = {7},
  number  = {4},
  pages   = {366},
  year    = {2019}
}

@article{turan1941external,
  title={On an extremal problem in graph theory (in {H}ungarian)},
  author={Tur{\'a}n, Paul},
  journal={Mat. Fiz. Lapok},
  volume={48},
  pages={436--452},
  year={1941}
}
\end{document}